\documentclass[12pt]{article}
\usepackage{amssymb,amsmath,enumerate,color,mathtools}
\usepackage[margin=33mm]{geometry}
\newtheorem{theorem}{Theorem}[section]
\newtheorem{lemma}[theorem]{Lemma}

\newtheorem{observation}[theorem]{Observation}
\newtheorem{conjecture}[theorem]{Conjecture}

\newenvironment{proof}{{\bf Proof.}}{\hfill \proofbox \vskip0.2cm }
\newenvironment{proofofL}{{\bf Proof of Lemma~\ref{sparselemma}.}}{\hfill \proofbox \vskip0.2cm } 
\newcommand{\proofbox}{\hbox{\vbox{\hrule\hbox{\vrule\phantom{\vrule height 8pt width 5pt depth 0pt}\vrule}\hrule}\quad}}
\newcommand{\eps}{\varepsilon}                       
\renewcommand{\epsilon}{\varepsilon}

\DeclareMathOperator{\Pee}{\mathbf{Pr}}

\DeclareMathOperator{\Bin}{Bin}      

\renewcommand{\geq}{\geqslant}
\renewcommand{\leq}{\leqslant}
\renewcommand{\ge}{\geqslant}
\renewcommand{\le}{\leqslant}
\DeclarePairedDelimiter\ceil\lceil\rceil
\DeclarePairedDelimiter\floor\lfloor\rfloor
                          
\makeatletter
\renewcommand\section{\@startsection {section}{1}{\z@}%
                                   {-3ex \@plus -1ex \@minus -.2ex}%
                                   {2ex \@plus.2ex}%
                                   {\normalfont\large\bfseries}}
\renewcommand\subsection{\@startsection{subsection}{2}{\z@}%
                                     {-2.5ex\@plus -1ex \@minus -.2ex}%
                                     {1.5ex \@plus .2ex}%
                                     {\normalfont\normalsize\bfseries}}
\renewcommand\subsubsection{\@startsection{subsubsection}{3}{\z@}%
                                     {-2ex\@plus -1ex \@minus -.2ex}%
                                     {1ex \@plus .2ex}%
                                     {\normalfont\normalsize\bfseries}}
 \renewcommand\paragraph{\@startsection{paragraph}{4}{\z@}%
                                    {1.5ex \@plus.5ex \@minus.2ex}%
                                    {-1em}%
                                    {\normalfont\normalsize\bfseries}}
\renewcommand\subparagraph{\@startsection{subparagraph}{5}{\parindent}%
                                       {1.5ex \@plus.5ex \@minus .2ex}%
                                       {-1em}%
                                      {\normalfont\normalsize\bfseries}}
\makeatother

\title{\bf \Large A Variant of the Erd\H os-S\' os Conjecture} 
\author{
Fr\'ed\'eric Havet\footnote{CNRS, Projet COATI, I3S (CNRS and UNS) UMR7271 and INRIA, Sophia Antipolis, France. Research supported by ANR under contract STINT ANR-13-BS02-0007
 (\texttt{Frederic.Havet@cnrs.fr}).}
\quad
Bruce Reed\footnote{School of Computer Science, McGill University, Montr\'eal, Canada (\texttt{breed@cs.mcgill.ca}). CNRS, Projet COATI, I3S (CNRS and UNS) UMR7271 and INRIA, Sophia Antipolis, France  (\texttt{reed@i3s.unice.fr}). Visiting Research Professor, ERATO Kawarabayashi Large Graph Project, Japan.} 
\quad
Maya Stein\footnote{Department of Mathematical Engineering and Center for Mathematical Modeling (UMI 2807 CNRS), Universidad de Chile, Santiago, Chile (\texttt{mstein@dim.uchile.cl}). Research supported by  CONICYT + PIA/Apoyo a centros cient\'ificos y tecnol\'ogicos de excelencia con financiamiento Basal, C\'odigo AFB170001, by FONDECYT Regular Grant 1183080, and by Millennium Nucleus Information and Coordination in Networks.}
\quad
David R. Wood\footnote{School of Mathematics, Monash University, Melbourne, Australia (\texttt{david.wood@monash.edu}). Research supported by the Australian Research Council.}}
\begin{document}
\maketitle

\begin{abstract}
A well-known conjecture of Erd\H{o}s and S\'os states that every graph with average degree exceeding $m-1$ contains every tree with $m$ edges as a subgraph.  We propose  a variant of this conjecture, which states  that every graph of  maximum degree exceeding $m$ and minimum degree  at least $\floor{ \frac{2m}{3}}$  contains every tree with $m$ edges. 

As evidence for our conjecture we show (i)  for every $m$ there is a $g(m)$ such that the weakening of the conjecture obtained by replacing  the first $m$ by $g(m)$ holds, and (ii) there is a  $\gamma>0$ such that the weakening of the conjecture obtained by replacing 
$\floor{ \frac{2m}{3}}$ by $(1-\gamma)m$ holds.

\end{abstract}

\section{Introduction}

A recurring topic in extremal graph theory is the use of degree conditions (such as minimum/average degree bounds)  on a graph to prove that it contains certain subgraphs.  For instance, every graph of minimum degree exceeding  $m-1$ contains a copy of each tree with $m$ edges. (Embed the root anywhere, and greedily continue embedding vertices whose parents are already embedded.)

In 1963, Erd\H os and S\' os conjectured the following  strengthening  of this fact:  if a graph has {\it average} degree exceeding $m-1$ then it contains every tree with $m$ edges as a subgraph. Their conjecture  has attracted a fair amount of attention over  the last decades. Partial solutions are given in~\cite{bradob, Haxell:TreeEmbeddings, sacwoz}, and in the early 1990's, Ajtai, Koml\'os, Simonovits and Szemer\' edi  announced a proof of this result for sufficiently large $m$. 
  In order to see that the Erd\H os-S\' os conjecture is best possible, observe that no $(m-1)$-regular graph contains the star $K_{1,m}$ as a subgraph. Alternatively, consider a graph that consists of several disjoint copies of the complete graph $K_{m}$; this graph contains no tree with $m$ edges as a subgraph.

The related Loebl-Koml\'os-S\'os  conjecture from 1995~\cite{EFLS95} states that if a graph has median degree at least $m$ then it contains every tree with $m$ edges  as a subgraph. This conjecture had also received considerable attention~\cite{AKS95,Cooley08,HlaPig:LKSdenseExact,PS07+, Z07+}, and recently, an approximate version was shown in~\cite{cite:LKS-cut0,cite:LKS-cut1,cite:LKS-cut2,cite:LKS-cut3} (see also~\cite{LKS:overview}). Note that the  examples  above demonstrate that  the Loebl-Koml\'os-S\'os  conjecture is  tight as well.

\smallskip

In this paper we propose a new conjecture for tree embeddings under degree assumptions.  We consider  the minimum and maximum degrees rather than the  average or median degrees. 

\begin{conjecture} 
\label{conj1}
If a graph has maximum  degree at least $m$ and minimum degree at least $\floor{\frac{2m}{3}}$ then it contains every tree
with $m$ edges as a subgraph.
\end{conjecture}

We remark that every graph of average degree exceeding $m-1$ has a subgraph of  minimum degree at least~$\frac{m}{2}$. Thus, replacing $\floor{\frac{2m}{3}}$ by $\frac{m}{2}$ in our conjecture would give a strengthening  of the Erd\H os-S\' os conjecture. However,  two simple examples show that the value  $\floor{\frac{2m}{3}}$ here is best possible. In both examples we consider the tree $T$ with $3k+1$ vertices obtained  from three stars on $k$ vertices by adding a new vertex $v$ adjacent to their centers. In the first example, $G$ is the graph obtained from two copies of $K_{2k-1}$ by adding a universal vertex. In the second  example, $G$ is the graph obtained by adding a universal vertex to $K_{2k-2,2k-2}$.

Nevertheless, focussing on the  minimum degree  of the graphs in question, could be an effective technique for approaching the Erd\H os-S\' os conjecture. Indeed, it might be possible to prove a natural common generalization of Conjecture~\ref{conj1}
and the Erd\H os-S\' os Conjecture which makes no mention of the average degree.

Note that Conjecture \ref{conj1} holds for paths (even with the weaker bound of $\frac m2$ on the minimum degree), because of the well-known Dirac-type result that every connected graph $G$ of minimum degree $\delta (G)$ has a path on  $\min\{2 \delta (G)+1,|V(G)|\}$ vertices. It also holds for trees with many leaves (see below).

As further evidence for Conjecture \ref{conj1}, we prove the following two weakenings.
 
\begin{theorem} 
\label{maint1}  There is a function $g$ such that  if a graph   has maximum  degree at least   $g(m)$ and minimum degree at least $\floor{\frac{2m}{3}}$ then it contains every tree $T$ with $m$ edges as a subgraph.
\end{theorem}

\begin{theorem} 
\label{maint2}
There is a $\gamma>0$ such that if a graph   has maximum  degree at least   $m$ and minimum degree at least $(1-\gamma)m$ then it contains every tree
$T$ with $m$ edges as a subgraph.
\end{theorem}



\smallskip

After proving some useful results on trees in Section~\ref{sec:trees}, we prove Theorem \ref{maint1}  in Section \ref{secmaint1} and Theorem \ref{maint2} in Section \ref{secmaint2}. While the proof of the first theorem is not very hard, the proof of the second theorem is much more complicated.
 We dedicate the remainder of the introduction to a sketch of some of the ideas used in both our proofs. For a more detailed sketch of the proof of Theorem \ref{maint2} we refer the reader to the beginning of Section~\ref{secmaint2}.

\medskip

Let us start with an easy observation that involves trees having a vertex $s$ that is adjacent to many leaves. We can embed $s$ in a  
maximum-degree vertex  $f(s)$ of the host graph, and then embed the rest of the tree, except for the leaves adjacent to $s$, in a  greedy fashion. Finally, we embed the leaves at $s$, exploiting the large degree of $f(s)$. Note that this procedure gives a proof of both our theorems, and of Conjecture~\ref{conj1}, for all trees that contain a  vertex adjacent to at least $\ceil{\frac{m}{3}}$ leaves.  It also proves Theorem~\ref{maint2} for all trees that contain a  vertex adjacent to at least~$\ceil{\gamma m}$ leaves. In particular, this proves the conjecture and Theorem \ref{maint1} for trees having a vertex of degree at least $\ceil{\frac{2m}3}$, and Theorem~\ref{maint2} for trees having a vertex of degree at least $\ceil{\frac {(1+\gamma)m}{2}}$.

The proof of both of our theorems for the remaining trees splits into two cases   depending on whether or not the host graph $G$ has a small dense subgraph. To illuminate why small dense subgraphs are important let us now prove the conjecture for  host graphs which do not contain any connected subgraphs with $m+1$ vertices having average degree  at least 2, that is, host graphs of girth at least $m+2$. If we greedily embed  a  tree with $m$ edges in such a graph by  embedding the vertices in breadth-first order, treating all the children of each vertex as a consecutive block, we see that for every non-root vertex $s$ we have embedded, the girth condition ensures that its image~$f(s)$ is adjacent to  the image of exactly  one vertex of the tree (namely the parent of $s$). Since $s$ has degree at most $\floor{\frac{2m}3}$, we will be able to embed its children into the unoccupied neighbours of $f(s)$. So the greedy embedding strategy succeeds in graphs of girth at least $m+2$.

Without the girth condition imposed in the illustrating example in the previous paragraph, but still assuming that our graph is relatively sparse and has no dense subgraphs (this is the first of the two cases mentioned above), we can still show that only a  few vertices  have many occupied neighbours. Our approach in this case is to try and stay  away from these vertices when embedding the rest of the graph. 
In order to do so, we exploit the well known fact (see Section~\ref{sec:trees}) that every tree $T$ with $m$ edges 
 has a  vertex $z$ such that at most one component of $T-z$ has more than a third of the vertices of $T$, and if such a component exists, it has at most two thirds of the vertices of $T$. The same is true replacing `a third' with $\gamma$ and `two thirds' with $1-\gamma$. This means that we can split the components of $T-z$ conveniently into two sets, such that the one containing more vertices can be embedded greedily using the minimum degree of the host graph, while embedding $z$ into a maximum degree vertex. Now, for embedding the remaining vertices we need to stay away from the occupied vertices. 
In proving Theorem \ref{maint1} this is relatively easy to do because $f(z)$ has 
huge degree, and so we have a lot of flexibility when placing the neighbours of $z$. In proving Theorem \ref{maint2}, $f(z)$ may only have 
$m$ neighbours which makes things harder. In this case we need to be more careful during the first phase of the embedding.
Here, the higher minimum degree comes in handy.  

Turning to graphs with small dense subgraphs (the second case mentioned above), we only discuss   the proof of Theorem \ref{maint2} here,
as the approach taken in  the proof of Theorem~\ref{maint1} is fairly straightforward. In that proof, we focus on the densest subgraphs of the host graph with at most $m+1$ vertices. For every such maximum-density  subgraph $H$,
 if $H$ has minimum degree $d$ then every vertex outside of $H$ 
sees at most $d+1$ vertices of $H$. Furthermore,  because $H$ is small and 
dense it turns out that we can embed trees with significantly more than $d$ vertices in $H$.

So we can often embed significantly more than $d+\eps m$ vertices of the tree in $H$ and 
just slightly less than $(1-\eps)m-d$ in $G-H$ which has minimum degree at least 
$(1-\eps)m-d-1$. In order to do so, we  split the tree $T$, by determining a cutvertex $z$, and grouping the components of $T-z$ into two sets of components, $\mathcal C_1$ and $\mathcal C_2$, whose sizes fit well with our embedding plans. 

There are some further complications: we need to consider some extensions of these small dense graphs, some dense bipartite graphs, and a partition of the graph into such dense pieces.  For more on these difficulties, see Section~\ref{secmaint2}. We hope our description here is enough to give a flavour of the proof. 

Finally, we mention that recent work \cite{BPMS19,RS19a,RS19b} has partially confirmed 
Conjecture~\ref{conj1}.
 
\section{Some Properties of Trees}\label{sec:trees}

In this section we prove some useful results on trees. Our first aim is to find a relatively large stable set whose vertices have degree at most $2$ in the tree. The small degree of the vertices in this set means that when embedding into a small dense subgraph $H$,  we will be able to  embed them last, after (carefully) embedding the rest of the vertices, thereby embedding many more than $\delta(H)$ vertices into $H$. 
%
%
%

\begin{lemma}
\label{lemfindings2dist3}
Every rooted tree $T$ with at least two vertices contains a stable set $S_T$   of size $ \ceil*{ |V(T)|/6 }$ not containing the root such that:
\begin{enumerate}[(a)]
\item every  vertex in $S_T$ is a leaf, or a vertex of degree $2$ whose parent is also a non-root  vertex of  degree $2$, and
\item no child of a vertex in $S_T$ is the parent of some other vertex of~$S_T$.
\end{enumerate}
\end{lemma}

\begin{proof}
Letting $\ell$ be the number of  non-root leaves of $T$, we see that removing the root of $T$ and all vertices of degree greater than $2$ in $T$ splits the non-root vertices of degree 1 and 2 in $T$ into fewer than  $2\ell$ paths of with total number of vertices at least  $|V(T)|-\ell$. 
We can put every third vertex within each of these paths  into $S_T$, as long as we start with the second from the root. We can thereby ensure that 
$|S_T|\geq (|V(T)|-3\ell)/3$. On the other hand, we can simply put all the  non-root leaves of $T$ into~$S_T$, so $|S_T|\geq \ell$. The result follows.
\end{proof}
  
Also, it turns out that matchings in the tree we wish to embed can be useful when embedding into a dense subgraph. This is because we can embed matched vertices one right after the other, that way their embedding happens under almost identical circumstances (with respect to the `used' or `unused' parts of the host graph). In addition, for the first vertex of a matching edge we can choose an image with high  degree into some set we wish to use for the second vertex.
  
\begin{lemma} 
\label{newesttreelemma}
For every tree $T$, and every $1\leq \ell \leq |V(T)|/2$, either $T$ contains at least  $|V(T)|-2\ell +2$ leaves, or for every vertex  $v$ 
of $T$, there is a subtree of $T$ with $2\ell$ vertices which contains $v$ and has a perfect matching.
\end{lemma}

\begin{proof}
Consider a maximum subtree $T'$ of $T$ containing $v$ which has a perfect matching. If some component of $T-T'$ has at least two vertices, then adding  two adjacent vertices of this component to $T'$, including the one joined to $T'$ by an edge, contradicts the maximality of $T'$. So all vertices in $V(T-T')$ are leaves, and since $|V(T')|$ is even, the result follows.
\end{proof}

%
%

  Finally, we prove a much used observation that allows us to split the tree into subtrees whose sizes we can control.

\begin{observation}
\label{obssep} Let $T$ be a tree on $t$ vertices.
\begin{enumerate}[(a)]
\item\label{a} There is $z\in V(T)$ such that every component of $T-z$ has $t/2$ or fewer vertices.
\item\label{b} For any $t'<\frac{t}{2}$,  either every component of $T-z$ has fewer than 
$t'$ vertices or there is a vertex $v_{t'}$ of $T-z$ such that the component of $T-v_{t'}$ 
containing $z$ has at most  $t-t'$ vertices and every other component has fewer than  $t'$ vertices. 
  \end{enumerate}
\end{observation}

\begin{proof}
For~\eqref{a}, we  root the tree and let $z$ be the vertex furthest from the root
such that the subtree formed  by $z$ and its descendants contains at least 
half the vertices.  

For~\eqref{b},  we can assume there is a component $C_1$ of $T-z$ having at least $t'$ vertices.  We root the tree at $z$ and   let ${v_{t'}}$ be the vertex furthest from $z$ in $C_1$
such that the subtree formed  by ${v_{t'}}$ and its descendants contains at least 
$t'$   vertices.  
\end{proof}
 
A {\it separator} for a tree $T$ on $t$ vertices is a vertex $z$ such that each component of $T-z$ has at most $t/2$ vertices. Note that the choice of such a $z$ is unique or there are two such choices which are endpoints of an edge $e$ such that each component of $T-e$ contains $t/2$ vertices.

 \section{The Proof of Theorem \ref{maint1}}
\label{secmaint1}

Define  $g(m):=(m+1)^{2m+6}+1$ and consider  a counterexample  $(m,G)$ minimizing $|E(G)|$. Let $v$ be a vertex of 
$G$ of degree $g(m)$ and note that minimality implies that if $uw$ is an edge of 
$G-v$ then one of $u$ or $w$ has degree $\floor{\frac{2m}{3}}$.
Let $t=m+1$ and let  $t'=m-\floor{\frac {2m}{3}}$. We assume $t \geq 3$ (otherwise the proof is easy).
We split the proof into two cases as follows.

\smallskip 

{\bf Case 1: }
$G$  contains a $K_{t^3,t'}$. 

\smallskip 

 Let $A$ be the smaller side of this complete bipartite graph and let $B$ be the  larger side. Minimality implies that 
 every vertex in $B$ has degree $\floor{\frac{2m}{3}}$. 
 
Thus, for any vertex $b$ of $B$, there are fewer than $t^2$ vertices of $B$ 
which  are adjacent to $b$ or  have a common neighbour with $b$ which has degree at most $t$. 
So, we can choose a  stable subset $B'$ of $t$ vertices of  $B$, such that no vertex of degree less than $t$ sees two vertices of $B'$. 

We take any subtree of  $T$ with $ 2t'+1$ vertices and embed it in $A \cup B'$ using more  vertices of $B'$ than of $A$.  
We can now greedily complete the embedding, since by the choice of $B'$, every (used or unused) vertex of degree less than $t$ has degree at least $\floor{2m/3} -1=m-(t'+1)$ into $G-B'$, while at least $t'+1$ vertices are already embedded into $B'$. 
\vskip0.3cm

{\bf Case 2: }
$G$  contains no 
$K_{t^3,t'}$. 

\smallskip
 
Note that in this case,  for any subset $S$ of $V$ that contains at least $t'$ vertices, we have 
\begin{equation}\label{nolargebip}
\text{less than  $t^3 \binom{|S|}{t'}$ vertices of $G$ see $t'$ or more vertices of $S$. }
 \end{equation}
 
Applying Observation \ref{obssep} with our chosen value of $t'$, we see that there is  a vertex $w$ of $T$  such that no component of $T-w$ has more than $\frac{2m}{3}$ vertices,
and all but the largest component have fewer than  $t'$ vertices. 

We embed $w$ into $v$. We greedily embed the largest component of $T-w$ into 
$G$. We then embed the remaining components of $T-w$, which have size at most $t'-1$. Whenever we come 
to embed such a component $K$, we proceed as follows.

Let $A_0$ be the set of vertices into which we have 
already embedded a vertex of $T$ (before starting to embed $K$). Successively, for $i=1,\ldots,t'$, let $A_i\subseteq V(G)-\bigcup_{j<i}A_j$ consist of all those vertices that have degree less than $t'-1$ in $G-\bigcup_{j<i}A_j$.  Note that $A_0,\dots,A_{t'}$ are pairwise disjoint. 

Now, each vertex of $A_1$ has degree at least  $\lfloor 2m/3 \rfloor-(t'-2)\ge t'$ into $A_{0}$.  Hence, we can use~\eqref{nolargebip} to see that  $$|A_1|\leq t^3 \binom{|A_0|}{t'}.$$ For $i\geq 2$, note that if $v\in A_i$, then (since $v\notin A_{i-1}$), we know that $v$ has a neighbour in $A_{i-1}$. So for $i\geq 2$, the definition of $A_{i-1}$ gives that $|A_i|\leq (t'-2)|A_{i-1}|$. Therefore,
$$|\bigcup_{i=0}^{t'} A_{i}|\leq \sum_{i=0}^{t'}(t'-2)^{i}t^3 \binom{|A_0|}{ t'} \leq t^{2t+4}.$$

So, since $g(m)>t^{2t+4}$, there is a neighbour 
 of $v$ outside of  $\bigcup_{i=0}^{t'} A_{i}$ in which we can embed the neighbour $x$ of $w$ in $K$. We now use the degree condition on the sets $A_i$ to greedily embed $K$ levelwise, allowing each level~$j$ (that is, the $j$th neighbourhood of $x$) to use vertices in $G-\bigcup_{i=0}^{t'-j} A_{i}$. This way we  ensure that $A_0$ is not used for our embedding of $K$. 

Iterating this process for each yet unembedded component of $T-w$ proves Theorem~\ref{maint1}.

\section{The Proof of Theorem \ref{maint2}}
\label{secmaint2}

Let us start by giving an overview of our proof.
Our proof has five parts.  In the first part, in Subsection~\ref{locsparse}, we show that if all the subgraphs of the host graph $G$
with at most $m+1$ vertices are really sparse,  then we can find the desired 
embedding (this is done in Lemma~\ref{sparselemma}). Thus we can assume that~$G$ has a  subgraph of size 
at most $m+1$ that is reasonably dense, i.e.~has average degree linear in $m$. 

 In Subsection~\ref{withoutverydense},   we show how to use such a  subgraph.  If we cannot find the desired embedding of $T$, then we  find a very dense subgraph $H$ of~$G$. More precisely, either $H$ has at most $m+1$ vertices, and is almost complete, in the sense that at every vertex there is only a small fraction of the possible edges missing, or $H$ is almost complete bipartite (in the same sense), with each of its sides having size at most $m$. 
 
Such a subgraph $H$ can be very useful for embedding a part of the tree, as its extreme density allows us to accommodate more vertices of $T$ than we would expect if we were only using the minimum degree. For technical reasons, it will be convenient to explain this approach in detail already in
 Subsections~\ref{fillingcomplete} and~\ref{fillingcompletebip} (before actually finding $H$ in Subsection~\ref{withoutverydense}). A series of lemmas given in these two subsections covers the range of possible situations that we might have to deal with in a later stage of the proof, when we wish to embed parts of $T$ into such a graph $H$.
 
 In the last part, in Subsection~\ref{finishing}, we put everything together.  We find a maximal set of disjoint very dense subgraphs $H_i$, knowing that at least one such subgraph is guaranteed to exist by what we said above. (Actually, our $H_i$ will be slight expansions of the subgraphs found in Subsection~\ref{withoutverydense}.) We show that if we cannot embed $T$ using the results of Subsections~\ref{fillingcomplete} and~\ref{fillingcompletebip},  there are only very few edges between the different subgraphs $H_i$, and between the union of the $H_i$ and the leftover of the graph $G$. By the results of Subsection~\ref{locsparse}, this means that there is no such leftover (as it would have to contain another very dense graph). Thus one of the $H_i$ contains a vertex of degree at least~$m$. Again making use of results of Subsections~\ref{fillingcomplete} and~\ref{fillingcompletebip}, we show we can embed $T$.
This completes the overview of our proof.

\medskip

We close this subsection with some preliminaries. We often  iteratively construct an embedding $f$ of  $T$ in such a way that the embedded  subtree  is always connected. In this case, whenever we come  to embed a vertex $s$ of $T$, there is a unique embedded neighbour $p(s)$ of $s$ and we need only ensure that $s$ is embedded in a neighbour of $f(p(s))$ which has not yet been used in the embedding. We refer to this as a {\it good} iterative construction process. 

Note that we can  and do assume that  no vertex  of $T$ is incident to more than $\gamma m$ leaves, as otherwise we can simply embed this vertex into a maximum degree vertex, greedily embed the tree except for the leaves incident to it and then embed these leaves. For this reason, all our lemmas are stated with this assumption.

\subsection{Locally Sparse Graphs}\label{locsparse}

A graph is {\it locally $m$-sparse} if it contains no subgraph with 
at most $m+1$ vertices and average degree exceeding $\frac{m}{25}$. 
The main result of this section is the following:

\begin{lemma}\label{sparselemma}
Suppose  $T$ is a tree with at most $m$ edges each of whose vertices is adjacent to 
  at most $\frac{m}{20}$  leaves and  $G$ is a locally $m$-sparse graph of minimum degree at least 
$\frac{19m}{20}$.  Then for any vertex $w$ of $G$ and separator $z$ for $T$,
we can find an embedding $f$ of $T$ in $G$ such that $f(z)=w$. 
\end{lemma}

Once we have proved Lemma~\ref{sparselemma}, we can continue our proof only considering host graphs $G$ that are not locally $m$-sparse.
Before proving Lemma~\ref{sparselemma}, we show an auxiliary result.

%
%

\begin{lemma}\label{sparsefact}\
Let $G$ be a locally $m$-sparse graph of minimum degree at least~$\frac{19}{20} m$.  Then for any $S\subseteq V(G)$ with $|S|\leq m-1$ there is a set $S'\supseteq S$ such that $G-S'$ has minimum degree at least $\frac {m}2$ and $|S'|\leq  |S|+\frac m{20}$.
\end{lemma}

\begin{proof}
Assume for a contradiction that there is no such set $S'$. For any set $S'\supseteq S$ such that $|S'|\leq  |S|+\frac m{20}$,
 there is a vertex $a$ of $G-S'$ having degree less than $\frac{m}{2}$ in $G-S'$, and so more than $\delta(G) -\frac{m}{2}=\frac{9m}{20}$ neighbours in $S'$. In particular $|S|\geq \frac {9m}{20}$, and we can find a set $A$ with $\ceil{\frac{m}{20}}$ vertices 
 such that every vertex in $A$ has at least $\frac{9m}{20}$ neighbours in $S\cup A$. 
 (Find $A$ by successively adding suitable vertices). 
 We choose  any set $B\subseteq S$ of size $|A|-1$, and note that the set $(S - B)\cup A$ has at most $m$ vertices and  induces more than $(\frac{9m}{20}- |B|)|A|\geq \frac{ 8m^2}{400} =\frac{m^2}{50}$ 
 edges, and thus, its average degree is above~$\frac m{25}$. This contradicts $G$ being locally $m$-sparse.
 
\end{proof}

\begin{proofofL}
We let $F$ be the union of some of the components of $T-z$ which together have between $\frac m4$ and $\frac{m}{2}$ vertices. If we can embed $T-F$ into a set $f(V(T-F))$ that avoids at least $|F| -1 + \frac m{20}$ neighbours of $w$, then by applying  Lemma~\ref{sparsefact} to $S_0= f(V(T-F))$, we obtain a set $S'_0$ that avoids at least $|F| -1$ neighbours of $w$ and such that $G-S'_0$ has minimum degree $\frac{m}{2}$. Now $z$ is adjacent to at most $\frac{m}{20}$ leaves it is adjacent to at most $|F|-1$ vertices of $F$, and so we can embed all neighbours of $z$ in $F$ into $N(w) - S'_0$. Then, 
 since $|F| \leq m/2$,  we can extend this embedding greedily into an embedding of all of $T$. Hence, fixing any set $N\subseteq N(w)$ with $|N|=\ceil*{\frac{19m}{20}}$, it suffices to embed $T-F$ using at least $\frac {m}{10}$ vertices outside $N$. 

Choose any set $S\subseteq N+w$ containing $w$  of size $\ceil{\frac {2m}3}$, and consider the set $S'\supseteq S$ given by Lemma~\ref{sparsefact}. Then $|S'-w|\leq \frac {3m}4$ and the vertices outside $S'$ have degree at least $ \frac m2$ into $G-S'$. We now embed  into $N-S'$ either all or $|N-S'|$ of the neighbours of~$z$ in $T-F$, and then embed the corresponding components of $T-F-z$ greedily, trying to avoid $S'$ as much as possible. 
  (Since $|T-F-z|\leq \frac {3m}4 \leq \frac{19m}{20}$,  we can clearly embed all of these components in $G$.)

If we never used any vertex in $S'$, then, depending on how many neighbours of $z$ we embedded into $N-S'$, we embedded either  all of $T-F-z$ in $G-S'$, or at least $|N-S'|-\frac m{20}\geq \frac{m}{10}$ vertices of $T-F-z$ in $G-N$ (since by assumption, $z$ has at most $\frac m{20}$ leaf neighbours). Since $z$ is not in $N$, in the first case we found the desired embedding, and in the second case we can greedily continue to find it.

 So assume we used $S'$, and let $x$ be the first vertex we embedded there. Then, since we tried to avoid $S'$, the parent  of $x$ is embedded  in a vertex that has at least $\frac m2$ neighbours in $G-S'$ that have already been used for our embedding.
 At least $\frac{m}{2} - |N - S'|$ of these vertices are outside $N$. But  $|N - S'| \leq |N - S|\leq \frac{19m}{20} - \frac{2m}{3}$.
 Hence at least $\frac {m}{2}  - \frac {19m}{20} + \frac{2m}{3} > \frac{m}{10}$ vertices outside $N$ have been used for embedding $T-F$, which is as desired. We greedily continue our embedding of $T-F$.
\end{proofofL}

\subsection{Filling Small Almost Complete Subgraphs}\label{fillingcomplete}

We now prove some auxiliary lemmas which are  at the heart of the whole proof. 
They allow us to use  almost complete subgraphs $H'$ 
of the host graph~$G$ in order to embed some suitable subtree~$T'$ of $T$. The point of these lemmas is that $T'$ is allowed to be substantially larger than the minimum degree of~$H'$. 

\begin{lemma} \label{almostfilling}
Let $0<\eps<\frac1{200}$, let
 $H'$ be  a   graph  with minimum degree at least 
$(1-2\eps )(|V(H')|-1)$, and let $T'$ be a tree with $m'$  edges, rooted at $z$, with $(|V(H')|-1)/2 \leq m' \leq (1-\eps)(|V(H')|-1)$. If each vertex of $T'$ is incident to  at most  $\eps m'/2$ leaves, then we can embed $T'$ in $H'$, choosing any vertex as the image of $z$.
\end{lemma}

%

\begin{proof}
Using Lemma~\ref{lemfindings2dist3}, we find a stable set $S$ in $T'$  of $\ceil{\frac{m'+1}{6}}$  non-root vertices which are leaves or  vertices of degree 2 
whose parents are non-root vertices of degree 2, such that no child of a vertex  in $S$
is the parent of some other vertex of $S$. 
By the definition of $S$, for any vertex in $T'-S$ that has more than one child in~$S$, all its children in $S$ are leaves. Hence 
\begin{equation}\label{fewChildren}
\text{each vertex  in $T'-S$ has at most $\eps m'/2$ children in~$S$.}
\end{equation}

So, we can choose a set $S'\subseteq S$ with $|S|/2\leq |S'|\leq |S|/2 + \eps m'/2$ such that  no vertex of $S-S'$ is closer to $z$ than any vertex of~$S'$, and such that for any given vertex in $T'$, either all or none of its children in~$S$ belong to $S'$.

Since our assumption on $\eps$ ensures that $$|S-S'|\geq \frac{m'+1}{12}-\frac{\eps m'}2\geq 2\eps( |V(H')|-1),$$  the minimum degree of $H'$ is large enough to allow us to use a good iterative construction process to greedily embed the component of $T'-N(S-S')$  that contains $z$ and all of $S'$. (In particular, children of vertices in $S'$ are embedded, but vertices from $S-S'$ and their parents are not.)   We immediately unembed the vertices of $S'$.  
 Note that $|T'-S'|\leq (1-2\eps )(|V(H')|-1)$, so it is possible to greedily embed the remainder of $T'-S'$. However, our plan is to embed the remainder of $T'-S'$, 
 in a way that  the vertices of $S'$ can be embedded afterwards. So we  do it cautiously. 
  
Call a vertex $u\in V(H')$ {\it good} for a vertex $s\in S'$, if $u$ is adjacent to both of the images of the two neighbours of $s$ in $T'$. Let $Bad$ be the set of all vertices $u\in V(H')$ with the property that  
  \[
  \text{there are more than $\frac{|S'|}2$ vertices in $S'$ for which $u$ is not good.  }
  \]


Since for each vertex  $s\in S'$ there are at most $4\eps( |V(H')|-1)$ vertices $u\in V(H')$ that are not good for $s$, it follows that there are at most $4\eps (|V(H')|-1)|S'|$ pairs $s\in S',u\in V(H')$ such that $u$ is not good for $s$. Therefore,  
\begin{equation}\label{Badsmall}
|Bad|\leq 8\eps(|V(H')|-1). 
\end{equation}

We now proceed our  embedding of $T'-S'$ in the following manner. 
When we are about to embed any vertex $p$  which has one or more children in $S-S'$, we try to embed $p$ into a vertex with many unused  neighbours in $Bad$.
 Note that since each unused vertex of $Bad$ has at most  $2\eps (|V(H')|-1)$ non-neighbours,  
there are at most $4\eps (|V(H')|-1)$ vertices 
which see less than half of  the unused vertices of  $Bad$. 
So, as the image of the parent of $p$ has more than 
\begin{align*}
(1-2\eps)(|V(H')|-1)-|V(T')-\{p\}-S'|& \geq -\eps(|V(H')|-1)+|S'|\\ & \geq \frac{m'+1}{12}-\eps(|V(H')|-1)\\ & >4\eps (|V(H')|-1)
\end{align*}
 unused neighbours (here we use our upper bounds on  $\eps$ and $m'$), 
we can embed $p$ into a vertex which sees more than half of the unused vertices of $Bad$. We immediately embed all children of $p$ in~$S$  trying to embed as many as possible into unused vertices of
$Bad$. By~\eqref{fewChildren},  we will be able to embed all these children, unless $Bad$ has less than $\eps m'$ unused vertices. 
Hence as long as $Bad$ has less than $\eps m'$ unused vertices, we embed the vertices of $S-S'$ in $Bad$. But, since $|S-S'|\geq \frac{m'}{12}-\frac{\eps m'}2 \geq |Bad| - \eps m'$, eventually $Bad$ will have less than $\eps m'$ unused vertices. After that, we  embed all the vertices greedily. Doing so, when we finish the embedding of $T'-S'$, we have used up all but at most $\eps m'$ vertices  of $Bad$.

It remains to embed $S'$. Consider the auxiliary bipartite graph between~$S'$ and the set $U$ of the so far unused vertices in $H'$, i.e.~the graph that has an edge $su$ for $s\in S'$, $u\in U$, if $u$ is good for $s$. By Hall's theorem, if we cannot embed $S'$ in $H'$, then in the auxiliary graph there is a  (non-empty) set $W\subseteq S'$ whose neighbourhood is smaller than $|W|$. In other words,  there is a subset $U_W\subseteq U$ such that $|U_W|<|W|$, and no vertex in $U-U_W$ is good for any vertex in $W$.

Because of our assumption on the minimum degree of 
$H'$, we know that $|U-U_W|\leq 4\eps(|V(H')|-1)$. On the other hand, by the other assumptions of the lemma,
\begin{equation}\label{sizesUS'}
|U|\geq |S'|+\eps( |V(H')|-1),
\end{equation}
and thus  $$|S'-W|<|S'|-|U_W|\leq |S'|-|U|+ 4\eps( |V(H')|-1) \leq 3\eps (|V(H')|-1).$$ So $|W|\geq |S'|/2$ (because $|S'|\geq \frac{m'}{12}\geq 6\eps(|V(H')|-1)$), and therefore, $U-U_W\subseteq Bad$. Since  $U$ contains at most $\eps m'<\eps( |V(H')|-1)$ vertices of $Bad$ (as we used all other vertices of $Bad$ earlier), we deduce from~\eqref{sizesUS'} that $|S'|<|U_W|<|W|$, a contradiction.
So we can embed all of $S'$ as planned.
\end{proof}

Observe that in the previous proof, we could have embedded an even larger tree $T'$ in $H'$, if we knew that the set $Bad$ could be filled up completely during the middle stage of the embedding, when we try to put as many vertices of $S-S'$ as possible into $Bad$. In fact, the term $\eps (|V(H')|-1)$ from~\eqref{sizesUS'} (which comes from the assumption that $m'\leq (1-\eps)(|V(H')|-1)$) is only needed to make up for the unused vertices of $Bad$, in the inequality of the second-to-last line of the proof.

Under certain circumstances, we {\it can} fill up $Bad$ completely, or up to a very small fraction. This is the content of the next  two lemmas.

\begin{lemma} \label{reallyfilling}
Let $0<\eps<\frac1{200}$, let
 $H'$ be  a   graph  with $m'+1$ vertices of minimum degree at least 
$(1-2\eps )m'$,   and  let $v$ be a vertex of 
$H'$ which sees all of $V(H')-v$. 
If  $T'$ is a tree with at most $m'$ edges such that each vertex of $T'$ 
is incident to  at most  $\eps m'/2$ leaves, then we can embed $T'$ in $H'$. 
\end{lemma}

\begin{proof}
Clearly we can assume $T'$ is not a single vertex, so $\frac{\epsilon m'}{2} \ge 1$.
We repeatedly  subdivide an edge from a leaf until $T'$ has exactly $m'$ edges.
Clearly, it is enough to prove the result for such $T'$.

We will proceed very much  as in the proof of Lemma~\ref{almostfilling}, with two 
small differences. Firstly, we avoid $v$ in our embedding throughout the process.
As before,  we stop right before reaching the parents of vertices in $S-S'$, and then unembed the vertices from~$S'$. We define the set $Bad$ as in the proof of  Lemma~\ref{almostfilling}, and observe that $|Bad|\leq 8\epsilon |V(H')|$. The next step is a little different from the proof of  Lemma~\ref{almostfilling}: 
When embedding the rest of $T'-S'$,
every time we consider the parent of a vertex in $S-S'$ we are happy if we embed at least half of its children in vertices of 
$Bad$. Since we always embed in a vertex which sees half of $Bad$, if we fail, then the current parent $p$ has more 
children in $S-S'$ than there are vertices in $Bad$. In this case, we  embed $p$ in $v$ (this is possible as $v$ sees all of $V(H')-v$), and use up all the vertices of $Bad$ for embedding the children of $p$ in $S-S'$. Observe that we are bound to find such a vertex $p$, because $\frac{|S-S'|}2>|Bad|$. We then embed the rest of $T'-S'$ greedily.

Now continue as in  the proof of  Lemma~\ref{almostfilling}, and embed $S'$. Note that although in~\eqref{sizesUS'}, we only get $|U|\geq |S'|$ instead of 
$|U|\geq |S'|+\eps |V(H')|$, we compensate for this shortcoming by having filled up all of $Bad$. Namely, from $U-U_W\subseteq Bad$ we can deduce that $U=U_W$, and thus  obtain $|S'| = |U_W|<|W|$, a contradiction which shows that  we can embed all of $S'$ as planned.
\end{proof}

The next lemma goes one step further than the previous lemmas, embedding the tree in- and outside the dense subgraph.

\begin{lemma} \label{almostreallyfilling}
For sufficiently small  positive $\gamma$ the following holds. 
Let $T$ be a tree  with $m$ edges none of whose vertices is incident to more than $\gamma m$ leaves. 
Let   $H'$ be  a   subgraph of $G$   with at most $m+1+3\gamma m$ vertices such that 
(i)  both $H'$ and $G-H'$  have  minimum degree at least $m-3\gamma m$, and (ii) there is a vertex $v $  of $H'$ with  degree at least $m$ in $G$.
Then  we can embed $T$ in $G$.
\end{lemma}

\begin{proof} 
We can assume that   $v$ does not see $m$ vertices of $H'$, as otherwise we 
are done by applying Lemma~\ref{reallyfilling} to $N(v) \cap H'$. 
We let $a=m-|N(v)\cap H'|$, and note that $v$ has at least $a \ge 1$ neighbours outside of $H'$.
We embed a separator  $z$ for $T$ into $v$. 

If the  sum $s$ of the sizes of the $a$ largest components of $T-z$ is  at least $3\gamma m$, or if $T-z$ has less than $a$ components, then we 
can choose some subset of these components that has between $3\gamma m$ and $\frac{m}{2}$ 
vertices. 
We embed these components greedily in $G-H'$, putting neighbours of $z$ into neighbours of~$v$, and then embed the rest of $T$ 
greedily  in $H'$ (note we can do so because of condition (i) of the lemma), and are done. 

So assume from now on that $T-z$ has at least $a$ components and that 
\begin{equation}\label{s3gammat}
s \le 3 \gamma m.
\end{equation}
%
%

Letting $F$ be the union of the $a$ largest components, all the  components of $T-F-z$ have size at most 
$\frac{s}{a}$. In particular, 
\begin{equation}\label{onlyzlargedegree}
\text{no vertex of $T-F$ other than $z$ has degree exceeding
$\frac{s}{a}$.}
\end{equation}
 Also note that 
 \begin{equation}\label{s2a}
 s \ge 2a,
\end{equation}
 since there are at most $\gamma m$ 
singleton components of $T-z$, and by~\eqref{s3gammat}, these cannot be part of $F$. 

We embed $F$ into $G-H'$ and  proceed as in the proof of Lemma~\ref{reallyfilling},  to embed $T-F$ into $H' \cap N(v)$ with one important difference, which we explain momentarily. 

As before,  we stop right before reaching the parents of the vertices in $S-S'$, and then unembed the vertices from~$S'$. We define the set $Bad$ as in the proof of  Lemma~\ref{reallyfilling}, and observe that $|Bad|\leq 24 \gamma m$. We continue 
embedding the rest of $T'-S'$, and as in Lemma~\ref{reallyfilling}, 
every time we consider the parent $p$ of a vertex in $S$ we are happy if we embed at least  half of its children in vertices of 
$Bad$. Let us call such a parent $p$ a {\it happy parent}. Since we always embed in a vertex which sees half of the unused vertices of $Bad$, if we cannot embed at least half of the children of $p$ in $Bad$, then we can use up half the currently unused  vertices of $Bad$ by embedding  children of $p$. Let us call such a parent~$p$ an {\it unhappy parent}. 

Next, we 
determine the size of the set of unused vertices of $Bad$ at the end of this process. 
Observe that at least half of the vertices  of $S-S'$ with happy parents get embedded in $Bad$, and thus, at most $2|Bad|$ vertices of $S-S'$ can have happy parents. So, at least $\frac{m(1-3\gamma)}{12}-\gamma m-2|Bad|\geq\frac m{15}$ vertices of $S-S'$ have unhappy parents. Thus, by~\eqref{onlyzlargedegree} there are at least $\frac{am}{15s}$ unhappy parents.
 
 So, setting $r=\frac{m}{s}$, we see that the number of unused vertices of 
$Bad$ at the end of the process is at most $24\gamma rs2^{\frac{-ra}{15}}$. Since 
$a \ge 1$, and, by~\eqref{s3gammat}, $r$ is at least $\frac{1}{3\gamma}$, if  $\gamma$ is sufficiently small then there are at most $\frac{s}{4}$ unused vertices of $Bad$ left. 

Now, note we are only embedding $|T-z-F|= m-s$ vertices into $N(v) \cap H'$,  
and the size of $N(v) \cap H'$ is at least $m-a$, which by~\eqref{s2a} is at least $m-\frac{s}{2}$. 
 This means we have more vertices in which to embed 
than vertices we need to embed even if we throw the unused vertices of $Bad$
away. So, we can  continue as in  the proof of  Lemma~\ref{almostfilling}, and embed $S'$.
\end{proof}

\subsection{\hskip-.2cm Filling Small Almost Complete Bipartite Subgraphs}\label{fillingcompletebip}

This section has a similar aim as the previous section. Instead of small almost complete subgraphs, we now focus on small almost complete bipartite subgraphs of the host graph $G$. 

We chose to start this subsection with the following lemma, because of the strong similarities of its proof with the proofs from the previous subsection.

\begin{lemma} \label{reallyfillingbipartite2}
Let $0<\eps<\frac1{200}$, let $T'$ be a tree with $m'$  edges such that  each vertex of $T'$ 
has at most  $\frac{\eps m'}{2}$ leaf children. Let $(C,D)$ be the unique 
$2$-colouring of $T'$ with $|C| \le |D|$. 
Let 
 $H'=((A,B),E)$ be  a  bipartite graph of minimum degree at least 
$(1-3\eps )m'$ such that  both $A$ and $B$  have at most 
$\floor{ (1+\eps) m'}$ vertices,  $B$ has at least $|D|$  vertices, and 
$A$ contains a vertex $v$ which sees all of $B$. 
Then  we can embed $T'$ in $H'$. 
\end{lemma}


\begin{proof}
We can assume that $m' \ge \frac 2 \eps$ or the tree must be a singleton and we are done. 
Because of the minimum degree condition on $H'$, we can greedily embed $T'$ unless $|C| < 3\eps m'+1 \le 4 \eps m'$, so we assume this 
is the case. This implies that there are at least $(1-8\eps)m'$ leaves of $T'$ in 
$D$  (for this, observe that rooting $T'$ arbitrarily, every non-leaf vertex in $D$ has at least one child in $C$). We let $T''$ be $T'$ with these leaves removed. 
Our plan is to embed $C$ in $A$ and $D$ in~$B$, starting with $T''$.
  
We use a good embedding algorithm to  begin embedding $T''$ in $H'-v$, starting 
with a vertex of $C$. We pause the procedure the first time that the set 
$X$ of vertices embedded in $A$ has edges  to more than half of the vertices of 
$T'-T''$.  We let $L$ be the set of neighbours of $X$ in $T'-T''$. Note that $(\frac12 - 4\eps) m' \leq |L|< (\frac12 + \frac12\eps)m'$, by our assumption on the number of leaf children at each vertex.

Let $f(X)$ be the image of $X$. We assign each vertex $x$ of $X$ a weight $w_x$ which is the number of 
 vertices of $L$  it is incident to. For every $X'\subseteq X$, we set $w(X')=\sum_{x\in X'} w_x$. Note that $w(X)=|L|$.
Call a vertex $b\in B$ $bad$ if there is a set $X'\subseteq X$ with $w(X')\geq \frac{|L|}{2}$ such that $b$ has no neighbour in $f(X')$. We let $Bad$ be the set of all bad vertices of $B$.

We claim that $Bad$ contains at 
most $8\eps m'$ vertices. Indeed, otherwise every vertex from $f(X)$ sees more than half the vertices of $Bad$. Consider the graph we obtain from blowing up each of the vertices $f(x)\in f(X)$ to a set $f'(x)$ of size $w_x$ (together with all adjacent edges). Then it is still true that every vertex in the set $f'(X):=\bigcup_{f(x)\in f(X)}f'(x)$ sees more than half the vertices of $Bad$. So by double-edge counting we see that on average, each vertex of $Bad$ sees more than half of the vertices of $f'(X)$. Thus in the original graph, each vertex of $Bad$ sees, on average, a set $f(Y)$ with $w(Y)\geq \frac{|L|}{2}$. So there is at least one vertex in $Bad$ actually seeing such a set $f(Y)$, contrary to the definition of $Bad$.

We shall now attempt to embed the remaining vertices of $T'-L$ so that we will be able to 
apply Hall's Theorem to finish the embedding by embedding~$L$. 
For this, we embed the remaining $T'-L$ using all vertices of $Bad$, we proceed as follows. Embed the rest of  $T'-L$ in a  greedy fashion, with the precaution that whenever we embed a vertex of $T'-L$, we immediately embed  all of its leaf children. Also, we avoid $v$ for the time being.
As in the proof of  Lemma~\ref{reallyfilling},   we see that we can choose images for the vertices of $C$ that see at least half of the unused vertices of $Bad$. Then we can 
 embed half of the leaf children  of each vertex  $x$ of $C$ into vertices of $Bad$ until 
 we reach a vertex $c\in C$ which has more children than there are unused vertices of $Bad$. Since $|L|>2|Bad|$,  there is such a $c$. We embed $c$ into $v$ and fill up the unused 
 vertices of $Bad$ with the leaf children of $c$. We continue greedily to embed all of $T'-L$.
 Let $f(T'-L)$ be the image of $T'-L$.

Now, by Hall's Theorem, to embed $L$ in $B\setminus f(T'-L)$, it is sufficient to prove that  $w(X') \leq |N(X') \setminus f(T'-L)|$ for all subset $X'$ of $X$.
Let $X'$ be a subset of $X$.
If $w(X')\leq |L|/2$ then,  since $H$ has minimum degree at least $(1 - 3\eps)m'$, $N(X')\setminus f(T'-L)| \geq (1 - 3\eps)m' - m' + |L| \geq (\frac{1}{2} - 7\eps)m' \geq (\frac{1}{4} +\frac{1}{4} \eps)m'  \geq  |L|/2$. 
If $w(X')\geq |L|/2$, then since $Bad\subseteq f(T'-L)$, by definition of $Bad$, we have $N(X')\setminus f(T'-L) = B\setminus f(T'-L)$ and so  $|N(X') \setminus f(T'-L)|\geq |L|$, because $|B| \geq |D|$.
In both cases, $w(X') \leq |N(X') \setminus f(T'-L)|$. This completes the proof.
 \end{proof}

 The ideas for the proofs of the remaining lemmas in this (and the subsequent subsections) are substantially different (although we still use Hall's theorem). An important tool is Lemma~\ref{orderedparents}, which is needed for Lemma~\ref{reallyfillingbipartite} below, and also for Lemma~\ref{manyLeaves}
  of Section~\ref{withoutverydense}.

For Lemma~\ref{orderedparents}, let us introduce good orderings of parents.
For a tree $T$ and a subset $L$ of its leaves, consider the set of parents $P$ of  $L$.
Order the vertices of $P$ as $p_1,\dots,p_m$ so that $p_i$ has at least as many leaf 
children as~$p_{i+1}$. Call any such ordering a {\it good ordering} of $P$.

\begin{lemma}\label{orderedparents}
Let $G$ be  a   graph with 
$\delta (G)\geq \frac {9m}{10}$, and let $T$ be a tree with m edges  such that no vertex of $T$ is incident to  more than $\frac m6$ leaves. Let $L$ be a subset of the leaves of $T$ such that  $|L|\geq \frac {9m}{10}$. Suppose there is a good ordering $p_1,\dots,p_a$ of the parents $P$ of $L$, and an embedding of $T-L$ in~$G$ such that for each $i\leq \floor{ a/2}$, we have
\begin{equation}\label{vertexnsebipsubgrX}
\text{$|N(f(p_{2i-1}))\cup N(f(p_{2i}))|\geq m$.}
\end{equation}
 Then we can extend the embedding of $T-L$ to an embedding $T$ in $G$.
\end{lemma}

\begin{proof}
First of all, note that  since no vertex has more than $\frac m6$ leaf children, for any set $S\subseteq P$ containing at most one of $p_{2i-1}, p_{2i}$, for each $i\leq \floor{ a/2}$, 
\begin{equation}\label{eq10}
\text{there are at most $\frac m6$ more leaves under $S$ than under $P-S$,}
\end{equation}
where we write `leaves under $X$' for leaves that are children of vertices in $X$. 


We use Hall's theorem to show we can embed the vertices of $L$. For this, consider the auxiliary bipartite graph $H$ spanned between the set $P'$ that arises from blowing up the image of  each $p\in P$ to a set $A_p$ of size equal to the number of leaf children of~$p$, and the set of unused vertices in $G$. For $a\in A_p$, the  edge $ab$ is present if $p$ is adjacent to $b$.

A matching saturating  $P'$ shows we can complete the embedding, so assume there is no such matching. By Hall's theorem, there is a set $S'\subseteq P'$ with $|N_H(S')|<|S'|$.
Because of~\eqref{vertexnsebipsubgrX},  $S'$  can only  contain vertices from  one of $A_{p_{2i-1}}, A_{p_{2i}}$, for each $i\leq \floor{a/2}$, and so, by~\eqref{eq10}, we know that $|S'|\leq |P'-S'|+\frac m6$. Thus $|S'|\leq\frac{2m}{3}$. But, as $\delta (G)\geq \frac {9m}{10}$, and since for the embedding of $T-L$ we used at most $\frac m{10}$ vertices, it follows that $|N(S')|\geq |S'|$, a contradiction.
\end{proof}

We continue with an analogue of Lemma~\ref{almostfilling} for bipartite host graphs. For its proof, we will make use of  Lemma~\ref{orderedparents}.

\begin{lemma} \label{reallyfillingbipartite}
Let $0<\eps<\frac1{200}$, let
 $H'=((A,B),E)$ be  a  bipartite graph of minimum degree at least 
$(1-\eps )m'$ such that $A$ has at most 
$\floor{ (1+\eps) m'}$ vertices  and  $B$ has exactly this many vertices. 
 Let  $T'$ be a tree with $m'$  edges such that  each vertex of $T'$ 
has at most  $\frac{m'}{6} $ leaf children. Then  we can embed $T'$ in $H'$. \end{lemma}

\begin{proof}
We let $(C,D)$ be the unique $2$-colouring of $T'$ with $|C| \le |D|$. 
Because of the minimum degree condition on $H'$, we can greedily embed $T'$ unless $|C| < \eps m'+1$, so we assume this 
is the case.  Note that $|C|\geq 2$ as $T'$ is not a star. Thus, we obtain 
$|C| < 2\eps m'$. We also obtain  that the set $L$ of leaves of $T'$ in 
$D$ has size at least  $(1-2\eps)m'$ (for this, observe that rooting $T'$ at a vertex of $C$,  every non-leaf vertex in $D$ has at least one child in $C$). Set $T'':=T'-L$. 
We  will embed $C$ in $A$ and $D$ in $B$. Consider a good ordering  $c_1,\ldots ,c_a$ of the parents of leaves in $L$.
We want to embed   $T''$ using an embedding~$f$ such that  for every $i \le \floor{ \frac{a}{2}}$, we have $|N_B(f(c_{2i-1})) \cup N_B(f(c_{2i}))| \ge m'$. 
Then we are done with Lemma~\ref{orderedparents}. 

As we embed $T''$, when we embed a vertex $c=c_i$  of $C$  paired with a vertex  $c'=c_{i\pm 1}$ which is already embedded, we choose as $f(c)$ an unused vertex  with the largest number of neighbours in $B-N(f(c'))$. 
Let us next estimate how large this number of neighbours will be.

Note that $|N(f(c')) \cap B| \ge (1-\eps)m'$ (by the minimum degree condition on $H'$), and so, we have
$|B-N(f(c'))| \le 2\eps m'$ (by our assumption on  the size of $B$). Also, each vertex in $B-N(f(c'))$ misses at most 
$2\eps m'$ vertices of $A$ (again by the minimum degree condition). Therefore, straightforward double-counting of non-edges between $A$ and $B-N(f(c'))$ gives that there is a set $A'\subseteq A$ containing  at least half the vertices of $A$ such that each vertex in $A'$ misses at most 
$$\frac{4\eps^2 (m')^2}{|A|/2} \le 16\eps^2 m'\le \eps m' $$ vertices of $B-N(f(c'))$. (For the first inequality, observe that $m'\leq 2|A|$ because of the minimum degree condition.) 

So, since we only embed
$|C| \ll |A'|$ vertices in $A$, we will be able to choose an image $f(c)$ that sees all but at most $\floor{\eps m'}$
 vertices of $B-N(f(c'))$. Then, $|N_B(f(c)) \cup N_B(f(c'))| \ge m'$ as desired. We thus find  the desired embedding of 
$T''$ and hence of $T$. 
\end{proof}

The next lemma is an analogue of Lemma~\ref{almostreallyfilling} for bipartite graphs.

\begin{lemma} \label{reallyfillingbipartite3}
Let $0<\eps<\frac1{200}$ and let 
 $H'=(A,B)$ be  a  bipartite subgraph of a graph $G$. Suppose $H'$ has minimum degree at least 
$(1-\eps )m'$, both $A$ and $B$ contain at most $(1+\eps)m'$ vertices, $A$ contains a vertex $v$ which has degree at least
$m'$ in $G$, and every vertex of $G-H'$ sees at least $(1-2\eps) m'$ vertices of $G-H'$. 
 Let  $T'$ be a tree with $m'$  edges such that  each vertex of $T'$ 
has at most  $\frac{\eps m'}{2} $ leaf children. Then  we can embed $T'$ in $G$.
 \end{lemma}

\begin{proof}
We let $(C,D)$ be the unique $2$-colouring of $T'$ with $|C| \le |D|$. 
Set $B':=N(v) \cap B$ and $a:=m'-|B'|$. Since we cannot embed the 
tree into~$(A,B')$ using Lemma~\ref{reallyfillingbipartite2}, we know that $|B'|<|D|$, and thus
 $$|C|=m'+1-|D| \le m'-|B'|=a.$$
We embed a separator $z$ for $T'$ into $v$. We will embed the leaf children
of~$z$ at the end of the process, which we can do because of our degree bound on $v$. 
Let $K_1,\ldots ,K_\ell$ be the non-singleton components of $T'-z$. Every $K_i$ contains a vertex of $C$, and thus $\ell \le a$. 

Since $z$ is a separator,  we know that 
\begin{equation}\label{discrep}
\big ||D \cap V(K_i)|-|C \cap V(K_i)|\big | \le \frac{m'-2}{2}
\end{equation}
 for all $i\leq\ell$. We let $w_i$ be the root of $K_i$, i.e.~the vertex of $K_i$ adjacent to $z$.
We will embed the roots $w_i$ into neighbours of $v$ in $G$  and then embed the rest of the tree  greedily in $H'$. 

First suppose that $v$ has at least $a$ neighbours in $A$.
Successively embed the roots $w_i$, in a way that ensures we can keep the embedding as balanced as possible at each step. This means that when we are about to embed $w_i$, we choose an image for $w_i$ in either $A$ or $B$, so that the larger colour class of $K_i$ will be forced to be embedded in that set among $A$, $B$ that when we finish our embedding will contain  less of $\bigcup_{j<i}V(K_j)$. (If both $A$, $B$  will contain the same number of vertices from $\bigcup_{j<i}V(K_j)$, for instance when $i=1$, we just arbitrarily choose either $A$ or $B$ for embedding $w_i$.)

Next, embed greedily the remainder of the components $K_i$. This can be done since the way we  embedded the roots $w_i$, together with~\eqref{discrep}, ensures that 
\begin{equation*}\label{discrepAll}
\big ||D \cap \bigcup_{j\leq i}V(K_j)|-|C \cap \bigcup_{j\leq i}V(K_j)|\big | \le \frac{m'-2}{2}
\end{equation*}
for each $i\leq \ell$.
Thus, throughout the embedding process of the $K_i$, we use at most 
$\frac{3m'}{4}$  vertices on each side $A$, $B$. 

Now, if $v$ has fewer than $a$ neighbours in $A$, we attempt to perform the same procedure.  If we run out of  neighbours of $v$ in $A$ during the embedding of the roots $w_i$,  then we start to embed roots $w_i$ which were to be embedded  into
$A$ into  $N(v)-H'$ (this is possible as $v$ has degree at least $m'$). We will embed the corresponding $K_i$ in $G-H'$, using the large minimum degree of $G-H'$. If at any point the total size of the  components embedded in $G-H'$ exceeds $\frac{m'}{4}$, then we stop embedding roots $w_i$ in $G-H'$. Instead, we embed the remaining $w_i$ in $B$ and the remaining $K_i$ in $H'$  (this is possible because of the minimum degree of $H'$). We will be able to embed the components whose roots are 
embedded in $G-H'$  because they have at most $\frac {3m'}{4}$ vertices and this graph has minimum degree at least $(1-2\eps)m'$. 
\end{proof}

\subsection{Graphs Without Very Dense Subgraphs}\label{withoutverydense}

The main result of this section is Lemma~\ref{goodlemma}. It says that if, in the situation of Theorem~\ref{maint2}, we cannot embed $T$ in $G$, then either $G$ is locally $m$-sparse (a situation we dealt with in Subsection~\ref{locsparse}), or $G$ contains at least one clique or bipartite $(m,\delta)$-dense subgraph (see below for the definition). In the Subsections~\ref{fillingcomplete} and~\ref{fillingcompletebip}, we saw how to use these subgraphs. Everything will be put together in the last part of our proof, in Subsection~\ref{finishing}.

 Let us now define the subgraphs we are looking for.
A subgraph $H$  of $G$ is 
{\it clique $(m,\alpha)$-dense} if it 
has at most $m+1$ vertices and minimum degree at least 
$(1-\alpha^{1/14})m$. 
A connected bipartite subgraph $H$ of $G$ is 
{\it bipartite $(m,\alpha)$-dense} if  it has minimum degree at least  $(1-\alpha^{1/14})m$ and each side of its  (unique) bipartition  has at most $m$ vertices.

We first treat the case that $T$ has many leaves. 
For this case, we need to make use of  Lemma~\ref{orderedparents} from Subsection~\ref{fillingcompletebip}.

\begin{lemma} \label{manyLeaves}
For every sufficiently small $\alpha>0$ the following holds. Suppose
 $G$ is   a   graph of minimum degree at least 
$(1-\alpha)m$ with no clique $(m,\alpha)$-dense subgraph and no bipartite $(m,\alpha)$-dense subgraph, and let $T$ be a tree with  at most m edges. If $T$  has at least  $(1-\alpha^{1/7})m$ leaves, but no vertex of $T$ is incident to  more than $\frac m6$ leaves, then we can embed $T$ in $G$.
\end{lemma}

\begin{proof}
Let $L$ be the set of leaves of $T$ and fix any good ordering $p_1,\dots,p_a$ of the parents of $L$. We claim that we can embed all of $T-L$ in $G$, via a good embedding $f$, while maintaining that, for each $i\leq \floor{ a/2}$, we have
\begin{equation}\label{vertexnsebipsubgr}
|N(f(p_{2i-1}))\cup N(f(p_{2i}))|\geq m.
\end{equation}
Then, Lemma~\ref{orderedparents} guarantees our partial embedding can be extended to an embedding of all of $T$. So we only need to prove we can find $f$ satisfying~\eqref{vertexnsebipsubgr}. 

For this, suppose that $p=p_j$ is the first vertex of $T-L$ that cannot be embedded without violating~\eqref{vertexnsebipsubgr}. Then there is an already embedded vertex $p'=p_{j\pm 1}$ such that the pair $p,p'$ violates~\eqref{vertexnsebipsubgr} for any embedding of~$p$.
Let $q$ be the parent of $p$, and let $A$ be a subset of size $\ceil{ (1-\alpha-\alpha^{1/7})m }$ of the unused neighbours of $f(q)$. (Note that there are that many unused neighbours of $f(q)$ because $|V(T-L)|\leq \alpha^{1/7}m+1$ by assumption.)
Let $B$ be the set of (used and unused) neighbours of $f(p')$. Since~\eqref{vertexnsebipsubgr} is violated for any embedding of~$p$, we know that $|B| \le m$ and that 
\begin{equation}\label{Alessthant}
\text{every vertex in $A$ has degree less than $m$.}
\end{equation}

Since $\delta (G)\geq (1-\alpha )m$, we have $|B|\geq (1-\alpha )m$, and also, since~\eqref{vertexnsebipsubgr} is violated,  every vertex of $A$ has at least $(1-2\alpha)m$ neighbours in $B$. So, there is a set $B'\subseteq B$ of size at least $(1-\sqrt{2\alpha})|B|$ such that each vertex in $B'$ has degree at least $(1-\sqrt{2\alpha})|A|$ into $A$. Note that $| B'| \geq (1-\sqrt{2\alpha})|B| \geq (1-\sqrt{2\alpha}) (1- \alpha) m \geq (1 - 2\sqrt{\alpha}) m$.

Assume for a contradiction that  $A-B'$ has size at most  $\alpha^{1/7}m$. Then every vertex of $A\cap B'$ has degree at least 
$$(1-\sqrt{2\alpha})|A| - \alpha^{1/7}m\geq (1-\sqrt{2\alpha}) ( 1- \alpha - \alpha^{1/7})m  - \alpha^{1/7}m \geq (1-\alpha^{1/14})m$$ 
in $G[A\cap B']$. Hence  $G[A\cap B']$ is clique $(m,\alpha)$-dense, a contradiction.

Hence $A-B'$ has size at least  $\alpha^{1/7}m$. Then $A\cap B'=\emptyset$, because the degree (in $G$) of any vertex $v\in A\cap B'$ would exceed 
\begin{eqnarray*}
|A\cup B'| - 2\alpha m -\sqrt{2\alpha}|A| & \geq & |B'|+\alpha^{1/7}m - 2\alpha m -\sqrt{2\alpha}m \\
& \ge & (1-2\sqrt{\alpha}) m  +\alpha^{1/7}m - 2\alpha m -\sqrt{2\alpha}m \\ & \ge &   m.
\end{eqnarray*}

contradicting~\eqref{Alessthant}. So,  
the bipartite subgraph   of $G$ with sides $A-B'$ and $B'-A$  is  bipartite $(m,\alpha)$-dense,
a contradiction. This proves the existence of an embedding satisfying ~\eqref{vertexnsebipsubgr}, completing our proof. 

\end{proof}

We now use Lemma~\ref{manyLeaves} together with Lemma~\ref{almostfilling} from the previous section to prove  the main result of this section: 

\begin{lemma} \label{goodlemma}
For every sufficiently small positive constant  $\alpha$, 
and $m \ge \alpha^{-2}$,  the following holds for each 
tree $T$ with at most $m$ edges none of whose vertices has more than   $\alpha m$ leaf children.  If
 $G$ is   a   graph of minimum degree at least 
$(1-\alpha)m$  that is not  locally $m$-sparse and contains neither a clique $(m,\alpha)$-dense 
subgraph nor a bipartite $(m,\alpha)$-dense subgraph then we can embed $T$ in $G$.
\end{lemma}


\begin{proof}
If $T$ has less than $m-1$ edges, then the tree obtained from $T$ by adding a path of length $m-|V(T)|$ on any vertex of $T$ also satisfies the hypothesis of the lemma. 
Thus, it suffices to prove the result for trees with $m-1$ or $m$ edges. 
Henceforth, we assume that $T$ has $m-1$ or $m$ edges.

We choose $\alpha$ small enough to satisfy certain inequalities in the proof. 

By Lemma~\ref{manyLeaves}, we may  assume that
\begin{equation}\label{fewLeaves}
\text{
 $T$  has  fewer than $(1-\alpha^{1/7})m$ leaves.  }
\end{equation}

We let $H$ be the densest subgraph of $G$ with at most $m+1$ vertices. We  let $\delta :=\delta (H)$ be the minimum degree of $H$, let $a$ be its average degree  and let  $w$ be some minimum 
degree vertex of $H$. Note that $a\geq \frac m{25}$, since $G$ is not locally $m$-sparse. So, as $\delta >\frac{a}{2}$ 
(by our choice of $H$), 
\begin{equation}\label{mt50}
\delta \geq \frac m{50}.
\end{equation}
Also,
\begin{equation}\label{yH}
\text{no vertex $y$ 
outside of $H$ sees more than $\delta +1$ vertices of $H$,}
\end{equation}
 as otherwise $H-w+y$ 
contradicts our choice of $H$. 
Furthermore, we can assume that 
\begin{equation}\label{m<}
\delta < (1-\alpha^{1/14})m
\end{equation}
 as otherwise $H$ is a clique
$(m,\alpha)$-dense subgraph.

We apply Observation \ref{obssep} to obtain a vertex $z$ such that the largest component
of $T-z$ has fewer than $m(1-\alpha^{1/3})$ vertices and every other component has fewer than
$\alpha^{1/3} m+1$ vertices. We let $F$ be a forest consisting of the union of some components 
of $T-z$ with between $\alpha^{1/3} m$ and $2\alpha^{1/3} m$ vertices. 
Note that since $z$ has at most $\alpha m$ leaf children (by the assumptions of the lemma), and since 
$\frac{|F| - \alpha m}{2} \geq \alpha m$
\begin{equation}\label{notmanynbsinF}
\text{
$z$ has at least $\alpha m$ non-neighbours in $F$.
}
\end{equation}
We embed $z$ into $w$ and the neighbours of $z$ in $F$ into $G-H$; this is possible because by \eqref{m<} $w$ has at least
$\delta(G) -\delta \geq (\alpha^{1/14} -\alpha)m\geq 2\alpha^{1/3}m$ neighbours in $G-H$. We leave the
remaining at least $\alpha m$ vertices of $F$ to embed at the end of the process.

By~\eqref{fewLeaves}, we know $T-F$ has fewer than $(1-\alpha^{1/7})m<(1-9\alpha^{1/3})m$ leaves. Hence, 
by Lemma~\ref{newesttreelemma}, we can choose a subtree $T'$ of $T-F$ 
containing $z$ which has  $2\ceil{ \alpha^{1/3} m} +2$ vertices and a perfect matching. 


As we are about to explain, we claim that either 
\begin{enumerate}[(i)]
\item there are $u,u'\in V(H)$ 
 such that  $d_H(u)\leq \delta  +3\alpha m$, and  $N_H(u')$ contains  a set $A$ of  $\ceil{ \delta -4\alpha^{1/3}m }$ vertices   each of which sees at  most $\delta +7\alpha^{1/3}m$  vertices of $H$ at least $\delta -4\alpha^{1/3}m$  of which are in  $N_H(u)$, or 
 \item   we can construct 
 an embedding of $T'$ so that for every   $x\in V(H)$ with $d_H(x)<\delta +3\alpha m$, we have used at least $3\alpha m$ vertices outside the closed  neighbourhood of $x$. 
 \end{enumerate}
 
 We will show that if (i) does not hold in $H$,  then we can find an embedding as in (ii).
 To do so, we root $T'$ at $z$  and consider a good iterative construction process for $T'$ into $H$ in which (a) we  embed the 
two vertices of each matching edge in consecutive iterations, and (b) we embed each vertex
$q$ in a randomly chosen unused  element of $N(f(p(q)))$.
Using our lower bound of $\alpha^{-2}$ on $m$, we shall prove that with positive probability for every   $x\in V(H)$ with $d_H(x)<\delta +3\alpha m$, we have used at least $3\alpha m$ vertices outside the closed  neighbourhood of $x$.

So consider a vertex $x$ such that $|N_H(x)|<\delta +3\alpha m$. Let us first estimate the probability that for a fixed  matching edge $e=\{v_1,v_2\}$ (of the perfect matching of $T'$) which does not contain $z$, we embed the second endpoint $v_2$ of $e$ outside~$N_H(x)$. For this, we  let $A$ be the set of all vertices that 
are neighbours of the image of $p(v_1)$ and 
see at  most $\delta+7\alpha^{1/3}m$  vertices of $H$ at least $\delta-4\alpha^{1/3}m$  of which are in  $N_H(x)$.
Since we assume (i) does not hold (for $u=x$ and $u'=f(p(v_1))$), we know that $|A|< \ceil{ \delta - 4\alpha^{1/3}m }$, 
while there are at least $\delta -|V(T')-v_1-v_2| \geq \delta -2\ceil{\alpha^{1/3}m}$ available possible images for the first endpoint $v_1$.
Thus, irrespective of the embedding to this point, 
the probability that~$v_1$ is embedded in a vertex outside $A$ is at least $2\alpha^{1/3}$. Therefore, again irrespective of the embedding to this point,  the  probability we embed $v_2$ outside $N_H(x)$ is  at least $4\alpha^{2/3}$. 
(For this, observe that every vertex outside $A$ has at least $4\alpha^{1/3}m$ neighbours in $H-N_H(x)$ and that at least $2\alpha^{1/3}m$ of them are unused.)
We have shown\footnote{We can decide for each matching edge $e$ when we come to it, whether or not its second endpoint is  in $N_H(x)$, and then choose the embedding of  its two endpoints conditional on our decision. We can make this decision by considering a random variable $z_e$ which is 1 with probability $4\alpha^{2/3}$.  If $z_e=1$ we do not put the second endpoint of $e$ in $N_H(x)$, otherwise we may or may not put this second endpoint in $N_H(X)$. The $z_e$ are independent.} that  the  number of non-neighbours of $x$ used in the embedding  is  a random variable  whose value dominates $\Bin(\ceil{\alpha^{1/3} m},4\alpha^{2/3})$, where the \emph{binomial random variable} $\Bin(n,p)$ is the sum of $n$ independent
0--1 random variables, each equal to $1$ with probability $p$.

Thus the probability that there are less than $3 \alpha m$ such non-neighbours  is bounded from above by the probability that $\Bin(\ceil{\alpha^{1/3} m},4\alpha^{2/3})$ is less than $3 \alpha m$. 
Chernoff's Bound (see~\cite{AlSp08,McD89}) states that for every $t\in[0,np]$,
\[\Pee\left(| \Bin(n,p)-np| > t\right)<2\exp\left(-\frac{t^2}{3np}\right)\,.\]
Hence the probability that the number of non-neighbours of $x$ used in the embedding  is less than $3 \alpha m$ is
less than $2\exp ( -\alpha m/12)$.


Since the number  of such vertices $x$ (vertices with  less than $\delta  +3\alpha m$ neighbours in~$H$) is at most~$m+1$, 
the probability that there is a vertex $x$ with $|N_H(x)|<\delta +3\alpha m$ such that less than $3 \alpha m$ non-neighbours of $x$ are used in the embedding is at most $(m+1) \times 2\exp ( -\frac{1}{12}\alpha m) \leq 2(m+1) \exp(-\frac{1}{12}m^{1/2})$ because $m\geq \alpha^{-2}$. Since we assumed~$m$ to be sufficiently large (since it is 
 at least $\alpha^{-2}$), this is less than $1$, and so there is an embedding as in (ii).

\smallskip

 If we find  an embedding as in (ii), then we can continue our good iterative construction process on the rest of $T-F$, always embedding in a  vertex  of~$H$ if possible. Clearly, we embed at least 
 $\delta +3\alpha m+1$ vertices in $H$. At this point, making use of~\eqref{yH}, we can greedily embed $F$ in the unused vertices of 
 $G-H$.
 
\smallskip

So we will from now on  assume that  (i) holds.  Then, we can find a subset~$B$ of  $\ceil{ \delta -1.5\alpha^{1/6}m}$ vertices of  $N_H(u)$ each of which sees at least  $\delta -7\alpha^{1/6}m$ vertices of~$A$. Indeed, otherwise there are at least 
$$1.5\alpha^{1/6}m\cdot (|A|-(\delta -7\alpha^{1/6}m))\geq 1.5\alpha^{1/6}m\cdot 3\alpha^{1/6}m\geq 4.5\alpha^{1/3}m^2$$
 non-edges between $A$ and $N_H(u)$, but the way $A$ was chosen allows for at most 
 $$|A|\cdot (|N_H(u)|-(\delta -4\alpha^{1/3}m))\leq \ceil{ \delta -4\alpha^{1/3}m} \cdot(4\alpha^{1/3}m+ 3\alpha m)<4.5 \alpha^{1/3}m^2$$
  such non-edges (here, we use that $\delta <m$ by~\eqref{m<} ).
Clearly every vertex of~$A$ sees at least $\delta-7\alpha^{1/6}m$ vertices of $B$.  

Let us recapitulate the situation as follows. We found sets $A, B\subseteq V(H)$ such that
\begin{equation}\label{recap1}
|A|=\ceil{ \delta -4\alpha^{1/3}m}, \ |B|=\ceil{ \delta -1.5\alpha^{1/6}m}
\end{equation}
and
\begin{equation}\label{recap2}
\text{the minimum degree from $A$ to $B$ and from $B$ to $A$ is  at least $\delta -7\alpha^{1/6}m$.}
\end{equation}

\medskip  
 
{\bf Case 1:}  $A- B$  and $B-A$ both have  size at  least $25\alpha^{1/6}m$.\smallskip

Let $[A-B,B-A]$ denote the bipartite subgraph of $G$ spanned by the edges between $A-B$ and $B-A$. Then
\begin{equation}\label{ABBA}
\text{$[A-B,B-A]$ has minimum degree at least $17\alpha^{1/6}m$.}
\end{equation}
Furthermore, each vertex of $A \cap B$ sees at least $|B|-7\alpha^{1/6}m+|A-B|-7\alpha^{1/6}m\geq |B|+11\alpha^{1/6}m$ vertices of $A \cup B$, and thus, 
\begin{equation}\label{AcapB}
\text{each vertex of $A \cap B$ sees at least $\delta +9\alpha^{1/6}m$ vertices of $A \cup B$.}
\end{equation}

By~\eqref{fewLeaves}, $T-F$ has fewer than $(1-\alpha^{1/7})m$ leaves, and by definition $|T-F|2\alpha^{1/3}$. Hence, $T-F$ has fewer than $|T-F|-33\alpha^{1/6}m$ leaves. 
So, by Lemma~\ref{newesttreelemma},  we can find a subtree $T^*$ of $T-F$ 
with $2\ceil{ 16\alpha^{1/6}m  }$ vertices which contains $z$ and has a perfect matching and hence a  2-colouring with colour classes  of equal size. Using~\eqref{ABBA}, we embed $T^*$ into  $[A-B,B-A]$, with $z$ in $A-B$. 

We claim that at this point, for every vertex $x$ of $A \cup B$ with  less than $\delta +\alpha^{1/6}m$ neighbours in 
$A \cup B$, 
\begin{equation}\label{atleast8}
\text{
we have embedded at least $8\alpha^{1/6}m$ vertices in non-neighbours of $x$.}
\end{equation}
For this, it suffices to observe that $x\notin A\cap B$ by~\eqref{AcapB}, and if $x\in A$, say, then  we embedded at least $16\alpha^{1/6}m$ vertices in $A-B$, but $x$ only sees at most $\delta +\alpha^{1/6}m-d_B(x)\leq 8\alpha^{1/6}m$ of these. (Here we used~\eqref{recap2} for the bound on $d_B(x)$.)

We continue embedding $T-F$ into $H[A \cup B]$ until we  have embedded at least 
$\delta +\alpha^{1/6}m+1$ vertices into it, which we can do because of~\eqref{recap2} and~\eqref{atleast8}.  
By definition of $F$, $z$ has at most $2\alpha^{1/3}m$ neighbours in $F$. We can embed these into 
$G-V(H)$, since $f(z)$  has at least $$\delta(G)-\delta > (1-\alpha )m- (1-\alpha^{1/14})m  \geq 2\alpha^{1/3}m$$ neighbours outside $H$ (we used~\eqref{m<} for the first inequality). We can then complete greedily the embedding of $T-F$ as at least $\alpha m$ vertices of $F$ have not yet been embedded by ~\eqref{notmanynbsinF}. Finally, complete the embedding of 
$F$ in $G-V(H)$; this is possible because in $G-V(H)$ every vertex  has degree at least $(1-\alpha) m - \delta -1$ by \eqref{yH},
and at most $m- \delta - \alpha^{1/6}m$ vertices of $T$ are embedded in $G-V(H)$.

\bigskip

{\bf Case 2:}  One of  $A-B$  or $B-A$ has size at most $25\alpha^{1/6}m$.\smallskip

Since by~\eqref{recap2}, each vertex of $A$ misses at most $7 \alpha^{1/6}m$ vertices of  $B$, and vice versa, 
 $G[A \cap B]$  has minimum degree at least  $|A \cap B|-7\alpha^{1/6}m$.  
 We consider a  largest induced subgraph $H'$ of $G$ with  at most 
 $m+1$ vertices and at most  $7\alpha^{1/6}m|V(H')|$ non-adjacent pairs of vertices, chosen so as to 
 maximize the number of edges in $H'$. So, if $H'$ has minimum degree $\delta'$ then 
 \begin{equation}\label{outsideH'}
\text{every vertex outside $H'$ has degree at most $\delta'+1$ in $H'$.}
\end{equation}

Note that since $G[A \cap B]$ is one possible choice for $H'$, 
\begin{equation}\label{t100}
|V(H')|\geq |A\cap B|\geq\min\{|A|,|B|\}-25\alpha^{1/6}m\geq \delta -27\alpha^{1/6}m>\frac{m}{100}, 
\end{equation}
where we used~\eqref{recap1} in the second-to-last inequality and~\eqref{mt50} in the last one. 
We obtain a subgraph $H^*$  of $H'$ by iteratively deleting vertices which are non-adjacent to 
more than $\frac{\alpha^{1/13} m}{3}$ vertices in the current subgraph. 
Then the minimum degree $m^*$ of $H^*$ is bounded by
\begin{equation}\label{H*mindeg}
\delta^*\geq |V(H^*)|-\frac{\alpha^{1/13}m}3.
\end{equation}
Clearly we delete at most 
$\frac{\alpha^{1/13} m}{10}$ vertices, that is, 
\begin{equation}\label{H'-H*}
|V(H')|-|V(H^*)|\leq \frac{\alpha^{1/13} m}{10}.
\end{equation}

If $|V(H^*)|$ exceeds $(1-\alpha^{1/13})m$ then as  $H^*$ has minimum degree at least $\delta^*\geq|V(H^*)|-\frac{\alpha^{1/13}m}3\geq (1-\alpha^{1/14})m$, we obtain that  $H^*$ is an $(m,\alpha)$-dense clique, contradicting our assumption that no such exist. So we can assume that
\begin{equation}\label{H*small}
|V(H^*)|\leq (1-\alpha^{1/13})m.
\end{equation}

Observation \ref{obssep} implies we can
choose a vertex $z^*$ of $T$ such that the largest component of  $T-z^*$ contains at most $(1-\frac{\alpha^{1/13}}{2})m$ vertices and every other 
component of $T-z^*$ contains fewer than $\frac{\alpha^{1/13}m}{2}$ vertices.
We choose a smallest possible forest $F^*$ consisting of the union of  components of 
$T-z^*$ whose total size is  between  $\frac{\alpha^{1/13}m}{2}$ and  $\alpha^{1/13}m$. 
We  note that since $z^*$ is incident to at most~$\alpha m$ leaves, 

\begin{equation}\label{Fandz*}
\text{$F^*$ contains at least
$\frac{\alpha^{1/13} m}{6}$ non-neighbours of  $z^*$. }
\end{equation}

First suppose $\delta'$ (the minimum degree of $H'$) is at most $|V(H')|-1-\frac{2\alpha^{1/13}m}3$.
We use a good iterative construction process to embed $T-F^*$ into $G$ with~$z^*$ in a vertex of $H^*$ and using vertices of $H^*$ when possible. 
By~\eqref{H*mindeg}, we use at least $|V(H^*)|-\frac{\alpha^{1/13}m}{3}+1$ vertices of $H^*$ before embedding any of $T-F^*$ outside~$H^*$. When we are about to first embed a vertex outside of $H^*$,  we proceed as follows. 

We start by embedding the neighbours of $z^*$ in $F^*$ into $G-V(H')$. Observe that this can be done, since because of~\eqref{Fandz*}, we know that $z^*$ has at most $\frac 56\alpha^{1/13}m$ neighbours in $F^*$, while $f(z^*)$ has at least $$(1-\alpha)m-|V(H^*)|-|V(H'-H^*)|\geq (\frac 9{10}\alpha^{1/13}-\alpha)m$$ neighbours in $G-H'$ (here, we used ~\eqref{H'-H*} and~\eqref{H*small}). Then we finish our embedding of $T-F^*$, just using the minimum degree of $G$. Finally, we embed  the  rest of 
$F^*$ in $G-V(H')$, using~\eqref{outsideH'},  our assumption on $\delta'$, and the fact that we used at least $|V(H^*)|-\frac{\alpha^{1/13}m}{3}+1$ vertices of $H^*$.

\medskip

So we can assume that $\delta'\geq |V(H')|-1-\frac{2\alpha^{1/13}m}3$.
Since $H'$ is not an  $(m,\alpha)$-dense clique, it follows that $|V(H')| \le m(1-\frac{\alpha^{1/14}}{2})$. 
We choose (the unique value of) $\epsilon$ such that  $\delta'=(1-2\epsilon)(|V(H')|-1)$.
Choose a subtree $T'$ of $T-F^*$ with 
$m'=(1-\epsilon)(|V(H')|-1)$ edges that  contains $z^*$ and subject to this has as few leaves as possible.
We note that this implies  if a vertex of $T'$ has two leaf children then all its leaf children are also 
leaves of $T-F^*$.

If no vertex of $T'$ is incident to more than $\frac{\epsilon m'}{2}$ leaves then 
Lemma~\ref{almostfilling}, with $\eps:=(\frac{\alpha^{1/13}}{3})(\frac{m}{|V(H')|-1})$,
ensures that  we can embed $T'$ in  $H'$ with $z^*$ embedded in a vertex of minimum degree in $H'$.  Note that for the application of  Lemma~\ref{almostfilling}, we use that  because of~\eqref{t100} and the fact that we can make $m$ as large as we want by making $\alpha$ small,  we know that $\frac{m}{|V(H')|-1}$ is at most $101$,  ensuring that $\epsilon$ is sufficiently small. When we stop there are at most $\epsilon m'$ unused vertices of $H'$. 

If some vertex of $T'$ is incident to more than $\frac{\epsilon m'}{2}$ leaf children then all but one of these 
leaves are also leaves of $T$. So,  by hypothesis, $\frac{\epsilon m'}{2} < \alpha m+2$. In this case,we just 
use a good iterative construction process to embed as much of  $T'$ into $H'$ as possible where to 
begin we embed $z^*$ in a minimum degree vertex of $H'$. When we stop
there are at most $2 \epsilon m' <4 \alpha m+8$ unused vertices of $H'$.

In either case, as above, we then embed  all of the neighbours of $z^*$ in $F^*$ into $G-H'$ which we can do because of our 
upper bound on the size of $|V(H)|$.  We then finish our embedding of $T-F^*$, just using the minimum degree of $G$. For the embedding of the rest of $F$, it is enough to observe that
 by  our chocie of  $H'$, every vertex of $V(G)-V(H')$ misses at least 
$max( 7\alpha^{1/6}m, 2 \epsilon m'-1)$ vertices of $H'$ and the number of unused vertices of $H'$ is at most 
$max(\epsilon m', \alpha m+8)$. 
\end{proof}

\subsection{ Finishing Things Off}\label{finishing}

In this section we prove Theorem \ref{maint2}.
We choose $\alpha < 1/200^{15}$  sufficiently small so that  Lemma~\ref{goodlemma} holds, and that other inequalities implicitly given in this section hold.  We choose 
$\gamma=\alpha^2$. Note that we can assume $m \ge \frac{1}{\gamma}=\frac{1}{\alpha^2}$  as otherwise the graph has minimum degree greater than $m-1$ so at least $m$, and we can just greedily embed $T$. 
We can also assume that 
no vertex has $\gamma m$ or more  leaf children as otherwise we can embed this vertex in a maximum degree vertex, greedily embed the tree except for its leaf children and then 
greedily embed these children. 

By Lemmas \ref{sparselemma} and  \ref{goodlemma},  we may assume $G$ contains
 a clique or bipartite $(m,\alpha)$-dense subgraph. For a clique $(m,\alpha)$-dense subgraph $D$ of $G$, by an {\it expansion of $D$} we mean a graph $H$ obtained by iterately adding vertices (one at a time) of $G-V(D)$ which see at least $(1-\alpha^{1/15})m$ vertices of the current expansion. For a bipartite $(m,\alpha)$-dense subgraph $D=(A,B)$ of $G$, by an {\it expansion of $D$} we mean a graph $H=(A',B')$ obtained by iterately adding one at a time vertices $v$ of $G-V(D)$ which see at least $(1-\alpha^{1/15})m$ vertices of one of the sides of the current expansion; we then add $v$ to the other side, and forget about all edges from $v$ to this side.  A {\it maximal expansion of $D$} is an expansion $H$ as defined above of maximal size.

%
%

 $G$ contains an expansion $H$ of a clique $(m,\alpha)$-dense subgraph with  $|V(H)|=1+\ceil{ (1-\alpha^{1/15})^{-1}m}$, then  we can embed $T$ 
within  it, by Lemma~\ref{almostfilling}, with $\eps:= \alpha^{1/15}< \frac{1}{200}$. (For this, observe that the minimum degree of $H$ is at least $\ceil{ (1-\alpha^{1/15})m} \geq  (1-2\eps)(|V(H)|-1)$, while the number of edges  of the tree $T$ is $m \leq (|V(H)| -1)(1-\eps)$.)
 So we can assume for all expansions $H$ of clique $(m,\alpha)$-dense subgraphs of $G$ we have
 \begin{equation}\label{maxexpsmall}
  |V(H)|<1+(1-\alpha^{1/15})^{-1}m.
  \end{equation}

Similarly, if $G$ contains an expansion $H=((A',B'), E')$ of a bipartite $(m,\alpha)$-dense subgraph $D=((A,B),E)$ with  
$\max \{|A'|, |B'|\} = \floor{ (1+\alpha^{1/15})m}$  then   we can embed $T$ 
within  it, by  Lemma~\ref{reallyfillingbipartite}. 
 So we can assume  for each expansion $H=((A',B'),E')$ of every bipartite $(m,\alpha)$-dense subgraph of $G$ we have
 \begin{equation}\label{maxbipexpsmall}
  \max\{|A'|,|B'|\} \le (1+\alpha^{1/15})m.
  \end{equation} 
 
 We will show below that if we cannot embed $T$, then for each maximal expansion $H$ of a clique or bipartite $(m,\alpha)$-dense subgraph of $G$, it holds that
 
\begin{enumerate}[(A)]
\item\label{lowdegintoDi}
no vertex of $G-H$  sees more than $2 \gamma m$ vertices of $H$, 
and
\item\label{lowdegoutofDi}
no vertex of $H$ sees more than $2 \gamma m$ vertices of $G-H$.
\end{enumerate}

Now, assuming~\eqref{lowdegintoDi} and~\eqref{lowdegoutofDi} hold, we consider a maximal sequence $D_1,\ldots,D_{\ell}$ of  clique and bipartite $(m,\alpha)$-dense subgraphs of $G$, together with corresponding maximal
expansions $H_1,\ldots,H_{\ell}$. More precisely, we choose $D_i$ as a clique or bipartite $(m,\alpha)$-dense subgraph of $G-\bigcup_{j<i} H_j$, 
and let $H_i$ be its maximal expansion in $G-\bigcup_{j<i} H_j$. Note that
$D_i$ is clique or bipartite $(m,\alpha)$-dense in $G$ and  by~(\ref{lowdegoutofDi}), applied to the graphs $H_j$ with $j<i$, we know that $H_i$ is also a maximal expansion of $D_i$ in $G$. 
 
 We will show below that moreover, if we cannot embed $T$, then
 
\begin{enumerate}[(A)]\setcounter{enumi}{2}
\item\label{lowdegintounionDi}
no vertex of $V(G)- \bigcup_{i=1}^{\ell} H_i$ sees more than $10 \gamma m$ vertices of $\bigcup_{i=1}^{\ell} H_i$.
\end{enumerate}
 
 Thus, if $G-\bigcup_{i=1}^{\ell} H_i$ is non-empty then by Lemmas \ref{sparselemma} and \ref{goodlemma}, we can embed $T$ within it. So, choosing a vertex $v\in V(G)$ of maximal degree, we can assume that $v$ is  contained in one of the $H_i$. 
 
 Now  if $v$ is in the expansion of a 
 bipartite $(m,\alpha)$-dense subgraph, then Lemma~\ref{reallyfillingbipartite3}, together with~\eqref{maxbipexpsmall} and~\eqref{lowdegintoDi}, tells us that  we can embed $T$. So we can assume $v$ is in the expansion $H$ of a clique $(m,\alpha)$-dense subgraph.   If $|V(H)|\leq 1+ (1+3 \gamma)m$, then Lemma~\ref{almostreallyfilling},  together with~\eqref{lowdegintoDi} and~\eqref{lowdegoutofDi}, gives an embedding of $T$ in $G$. So $|V(H)| \ge 1+(1+3 \gamma)m$.
 Setting  $\eps:=\frac{|V(H)|-1-m}{|V(H)|-1}$ we see that 
  ~\eqref{lowdegoutofDi} guarantees that the minimum degree of $H$ is at least $(1-\gamma)m-2\gamma m=(1-3\gamma)m\geq (1-2\eps)(|V(H)|-1)$. 
   Furthermore   our upper bound \eqref{maxexpsmall} on the size of expansions  ensures $\eps \le \alpha^{1/15}<\frac{1}{200}$.
  Finally  our lower bound on $V(H)$, ensures that for sufficiently 
 small $\gamma$,  $\eps =1 -\frac{m}{|V(H)|-1} \ge 
 1- \frac{1}{1+3\gamma} = \frac {3\gamma}{1+3\gamma} \ge 2 \gamma$.
 Hence we  can embed $T$ using Lemma~\ref{almostfilling}.   
 
This  completes the proof of the theorem.  
It only remains to show~\eqref{lowdegintoDi}, ~\eqref{lowdegoutofDi} and~\eqref{lowdegintounionDi}. 

\bigskip
 
 To prove \eqref{lowdegintoDi}  we consider the expansion $H$  of some clique or bipartite $(m,\alpha)$-dense subgraph $D$  of $G$.
Note that by the definition of an expansion, $G-H$ has minimum degree at least $(\alpha^{1/15}-\gamma)m$. Let $w$ be a vertex outside of $H$ with maximum degree  into $H$. Let $d$ be the number of its neighbours in $H$ and assume for a  contradiction that  $d>2\gamma m$.
 
 Our plan is to find an embedding of $T$ in $G$, here is an outline of the proof. We distinguish between two cases: First, we treat the case that $d$ is relatively large (almost $m/2$ or larger). In this case we embed a suitable vertex $z$ of $T$ in $w$, a few small components of $T-z$ outside of $H$, and the main part of $T$ in $H$. The other case is that $d$ is rather small (between $2\gamma m$ and almost $m/2$). In that case, we embed a suitable vertex $z$ of $T$ in $w$, and embed into~$H$ a set $\mathcal C$ of components of $T-z$ whose union contains a little bit more than $d$, namely $d+\gamma m$, vertices. This is possible since $T-z$ has at most $\gamma m$ singleton components, and so the number of neighbours of $z$ in $\bigcup\mathcal C$ is at most $d$. We then embed the rest of $T$ outside $H$.  Let us now turn to the details of this plan.
 
 \smallskip
 
{\bf Case 1:}  $d> (\frac{1}{2}-\frac{\alpha^{1/15}}{6})m$.\smallskip

In  this case we choose a vertex $z$ of 
 $T$ such that the largest component of $T-z$ has  at most 
 $(1-\frac{\alpha^{1/15}}{3})m+1$ vertices and every other component 
 has fewer than 
 $\frac{\alpha^{1/15}m}{3}$ 
 vertices (this is possible by Observation~\ref{obssep}). We embed~$z$ into $w$. 
 We choose  some components including all the 
 (at most $\gamma m$) singleton components, so that the union of these components has  
 between  $\frac{\alpha^{1/15}m}{3}$
 and $\frac{2\alpha^{1/15}m}{3}$ vertices. We embed these components greedily into $G-H$.
 Since the remaining components of 
  $T-z$ each have at least two vertices,  there are at most $\frac 12(1-\frac{\alpha^{1/15}}{3})m<d$ of them. We embed the roots (neighbours of~$z$) of these components into  neighbours of $w$ in $H$, preferring vertices 
  of $D$. 
  
  We then proceed to embed greedily into $H$ all those components of $T-z$ whose root was  embedded in $H-D$. If such components exist, then, since there are at most $4\alpha^{1/15}m$ vertices in $H-D$ (at most $((1-\alpha^{1/15})^{-1}+\alpha^{1/14})m \leq 3\alpha^{1/15}m$ if $D$ is clique dense by~\eqref{maxexpsmall}, and at most $4\alpha^{1/15}m$ if $D$ is bipartite dense by~\eqref{maxbipexpsmall}), and since we preferred vertices of $D$ for putting down the roots, we must have embedded at least $d-4\alpha^{1/15}m\geq \frac m3$ roots of other components into $D$. So, as we already got rid of singleton components, there are at least $\frac m3$ vertices in components whose root is in $D$   which we are not yet embedded. Thus, the minimum degree of $H$, which is $(1-\alpha^{1/15})m$, is sufficient for embedding all components with roots in $H-D$. Finally,  we embed all those components whose root was embedded 
  in $D$. For this, observe that being a $(m,\alpha)$-dense subgraph, $D$ has minimum degree $(1-\alpha^{1/14})m$, which is sufficient for embedding the rest of  $T$ (since at least $\frac{\alpha^{1/15}m}{3}$ vertices of $T$ were already embedded outside $H$). 
  
 \smallskip
  
{\bf Case 2:}  $2\gamma m<d \le  (\frac{1}{2}-\frac{\alpha^{1/15}}{6})m$.\smallskip

Then $G-H$ has minimum degree 
 \begin{equation}\label{minG-H} 
\delta (G-H)\geq m-d-\gamma m.
\end{equation}
 We choose a vertex  $z$  of $T$  such that 
 the largest component $C_{max}$ of $T-z$ has fewer than $1+m-d-\gamma m$ vertices 
 and every other component has  at most $d+\gamma m$ vertices (possible by Observation~\ref{obssep}), and embed $z$ into $w$.  Let $\mathcal C_{T-z}$ be the set of components of $T-z$.
Take a smallest set $\mathcal C\subseteq \mathcal C_{T-z} - \{C_{max}\}$ with 
    \begin{equation}\label{sizeC} 
d+\gamma m\leq |\bigcup_{C\in \mathcal C}V(C)|.
\end{equation}
Clearly,
    \begin{equation}\label{maxsizeC} 
 |\bigcup_{C\in \mathcal C}V(C)|\leq 2d+2\gamma m.
\end{equation}
We claim that moreover,
    \begin{equation}\label{maxsizeC-gamma} 
\text{ if $\mathcal C$ has singleton components, then } |\bigcup_{C\in \mathcal C}V(C)|\leq 2d-\floor{ \gamma m}.
\end{equation}
In order to see~\eqref{maxsizeC-gamma} suppose that  $|\bigcup_{C\in \mathcal C}V(C)|\geq 2d- \floor{ \gamma m} +1$. We need to show that $\mathcal C$ has no singleton components. For this, it suffices to observe that by the minimality of $\mathcal C$, for each component $C^
*\in\mathcal C$ we have that $|\bigcup_{C\in \mathcal C, C\neq C^*}V(C)|\leq \ceil{ d+\gamma m} -1$. 
So $$|V(C^*)|\geq  2d-\floor{ \gamma m} +1 - (\ceil{ d+\gamma m} -1)=d- \floor{ \gamma m}- \ceil{ \gamma m}+2>1,$$
where for the last inequality we apply  our hypothesis  that  $d>2\gamma m$.

Next, we wish to show that
\begin{equation}\label{howmany}
 |\mathcal C|\leq d.
 \end{equation}
If $\mathcal C$ has singleton components, then there are at most $\floor{\gamma m}$ such components, and~\eqref{howmany} follows from the fact that by~\eqref{maxsizeC-gamma},
$$|\mathcal C|\leq\floor{\gamma m} +\frac{|\bigcup_{C\in \mathcal C}V(C)|-\floor{\gamma m}}{2}\leq\floor{\gamma m} +
\frac{2d-2\floor{ \gamma m}}{2}= d.$$
If $\mathcal C$ has no singleton components, and additionally, $|\bigcup_{C\in \mathcal C}V(C)|\leq 2d$, then $|\mathcal C|\leq\frac{|\bigcup_{C\in \mathcal C}V(C)|}{2}\leq d,$ as desired, so let us now assume that $|\bigcup_{C\in \mathcal C}V(C)|> 2d$.
Then $|\mathcal C|\leq 3$, as otherwise the set $\mathcal C'$ obtained from $\mathcal C$ by deleting the smallest component satifies $|\bigcup_{C\in \mathcal C'}V(C)|>\frac 34\cdot 2d>d+\gamma m$ (since $d>2\gamma m$), contradicting the minimality of $\mathcal C$.
Moreover, since $d>2\gamma m\geq 2$, we know that $d\geq 3$. Thus again, $|\mathcal C|\leq d$. This completes the proof of~\eqref{howmany}.
 
 We now embed $T-z$.
 By~\eqref{minG-H} and by~\eqref{sizeC}, the minimum degree of $G-H$ is large enough to greedily embed into $G-H$ all the 
 components of $T-z$ that are not in $\mathcal C$. Next, we embed the (by~\eqref{howmany} at most $d$) roots of the components from   $\mathcal C$ into~$H$, as above preferring vertices in $D$ over vertices in $H-D$. We then embed all components whose root was put into $H-D$, and finally embed the components with root embedded in $D$. In order to see that we succeed in embedding all of $T$, we argue similarly as in the previous case: For the components with root in $H-D$, note that again, we must have embedded at least $d-4\alpha^{1/15}m$ roots of non-singleton components into $D$, so, unless $d-4\alpha^{1/15}m<\alpha^{1/15}m$, we can argue as  above that the minimum degree of $H$ is sufficient. On the other hand, if $d-4\alpha^{1/15}m<\alpha^{1/15}m$, that is, if $d<5\alpha^{1/15}m$, then by~\eqref{maxsizeC}, $$ |\bigcup_{C\in \mathcal C}V(C)|\leq 2d+2\gamma m< 10\alpha^{1/15}m+2\gamma m,$$ so again, the minimum degree of $H$ is sufficient.
 For the components with root in $D$, note that as above, the minimum degree of $D$ is sufficient for embedding them because by~\eqref{maxsizeC} at least $$m-|\bigcup_{C\in \mathcal C}V(C)|\geq m-2d-2\gamma m\geq \frac{\alpha^{1/15}m}3-2\gamma m\geq \alpha^{1/14}m$$ vertices of $T$ were already embedded outside $H$. This completes the proof of Case 2, and thus of~\eqref{lowdegintoDi}. 
 
 \bigskip
 
  To prove \eqref{lowdegoutofDi}  we consider the expansion $H$  of some clique or bipartite $(m,\alpha)$-dense subgraph $D$  of $G$.
 We let $w$ be a vertex of $H$ which has maximum degree $d_{G-H}(w)$ outside of $H$ 
 and set $$d:=\min\{ d_{G-H}(w),\alpha^{1/15}m\}.$$  Then, $H$ has minimum degree at least $m-d-\gamma m$ (this is clear if $d=d_{G-H}(w)$, and follows from the fact that $H$ is an expansion in the case that $d=\alpha^{1/15}m$). 
 We embed a separator $z$  for $T$ into~$w$. We choose a minimal set $\mathcal C$ of 
 components of   $T-z$  containing at least $d+\gamma m$ vertices. 
 Since at most $\gamma m$ 
 of these components are singletons, we need at most $\frac{d}{2}+\gamma m<d$
 components.
 Furthermore, their total size is at most $\frac{m}{2}$
 (as the size of the components of $T-z$ is bounded by this number, since $z$ is a separator). The minimum degree of $H$ is clearly enough to greedily embed  into $H$ all those  
 components of $T-z$ that are not in $\mathcal C$.  We then embed the components from~$\mathcal C$ into $G-H$. 
 We embed the (at most $d$) neighbours of $z$ first.  After that, the minimum degree of $G-H$ (which, by~\eqref{lowdegintoDi}, is at least $(1-\gamma-2\gamma)m\geq \frac m2\geq \bigcup_{C\in\mathcal C}|V(C)|$) ensures we can embed the remainder of the components from~$\mathcal C$. This completes the proof of~\eqref{lowdegoutofDi}. 
 
 \bigskip
 
 It remains to prove \eqref{lowdegintounionDi}. We shall do so by inductively proving  that if we cannot embed $T$, then for every $j$ between 1 and $\ell$, 
 
 \begin{enumerate}[(C')]
 \item no vertex of $V(G)- \bigcup_{i=1}^j H_i$ sees more than $10 \gamma m$ vertices of $\bigcup_{i=1}^j H_i$.\label{induc}
 \end{enumerate}
 
 For $j=1$, (C') holds by~\eqref{lowdegintoDi}. Assuming (C')  holds for $j-1$, let us show that (C') also holds for $j$. By~\eqref{lowdegintoDi}, no vertex of $V(G)- \bigcup_{i=1}^j H_i$ sees more than $10\gamma m+2\gamma m=12 \gamma m$ vertices of $\bigcup_{i=1}^j H_i$.
 Thus,  $G- \bigcup_{i=1}^j H_i$ has minimum degree at least $m-13\gamma m$. 
 
Suppose there is a vertex  $w\in V(G)- \bigcup_{i=1}^j H_i$ which sees at least $10 \gamma m$
 vertices in $\bigcup_{i=1}^j H_i$. Our aim to show that we can then embed $T$.
 We embed a separator $z$  for $T$ into $w$. We choose a minimal set $\mathcal C$ of 
 components of   $T-z$  containing at least $13 \gamma m$ vertices. 
 Since at most $\gamma m$ 
 of the components in $\mathcal C$ are singletons, we know that $|\mathcal C|\leq 7 \gamma m$.
 Furthermore, $\bigcup_{C\in\mathcal C}|V(C)|\leq\frac{m}{2}$.  We then embed the 
 components of $T-z$ that are not in $\mathcal C$ greedily into~$G- \bigcup_{i=1}^j H_i$, using the minimum degree of~$G- \bigcup_{i=1}^j H_i$. Finally, we embed the components from $\mathcal C$ into $\bigcup_{i=1}^j H_i$, embedding the (at most $7 \gamma m $) neighbours of $z$ first, and using the minimum degree of the $H_i$ for the rest of these components (this works since we embedded at least $\frac m2$ vertices outside of $\bigcup_{i=1}^j H_i$). This shows (C'), and thus  completes the proof of \eqref{lowdegintounionDi}.

 \subsection*{Acknowledgements}
 
The authors would like to thank the referees for their comments which greatly improved the readability and quality of the paper.

\newcommand{\etalchar}[1]{$^{#1}$}

\end{document}